\documentclass[12pt]{amsart}
\usepackage{amsrefs}
\usepackage{amssymb}
\usepackage{amsmath}
\usepackage{proof}
\usepackage{centernot}
\usepackage{enumerate}
\usepackage{graphicx}
\usepackage{fullpage}
\usepackage{blindtext}
\usepackage{cleveref}

\newbox\gnBoxA
\newdimen\gnCornerHgt
\setbox\gnBoxA=\hbox{$\ulcorner$}
\global\gnCornerHgt=\ht\gnBoxA
\newdimen\gnArgHgt
\def\quote #1{%
\setbox\gnBoxA=\hbox{$#1$}%
\gnArgHgt=\ht\gnBoxA%
\ifnum     \gnArgHgt<\gnCornerHgt \gnArgHgt=0pt%
\else \advance \gnArgHgt by -\gnCornerHgt%
\fi \raise\gnArgHgt\hbox{$\ulcorner$} \box\gnBoxA %
\raise\gnArgHgt\hbox{$\urcorner$}}

\title{On sequents of $\Sigma$ formulas}
\author{Andre Kornell}
\address{Department of Mathematics \\  University of California \\ Davis, CA 95616}
\email{kornell@math.ucdavis.edu}

\newtheorem{theorem}{Theorem}[section]
\newtheorem{lemma}[theorem]{Lemma}
\newtheorem{proposition}[theorem]{Proposition}
\newtheorem{corollary}[theorem]{Corollary}

\theoremstyle{definition}
\newtheorem{definition}[theorem]{Definition}

\theoremstyle{remark}

\theoremstyle{plain}
\newtheorem*{theorem*}{Theorem}
\newtheorem*{lemma*}{Lemma}
\newtheorem*{proposition*}{Proposition}

\theoremstyle{definition}
\newtheorem*{definition*}{Definition}

\theoremstyle{remark}

\def\dotminus{\mathbin{\ooalign{\hss\raise1ex\hbox{.}\hss\cr
  \mathsurround=0pt$-$}}}

\newcommand{\<}{\langle}
\renewcommand{\>}{\rangle}
\renewcommand{\:}{\colon}
\newcommand{\subsetof}{\subseteq}

\newcommand{\suchthat}{\,|\,}

\newcommand{\Union}{\bigcup}
\newcommand{\union}{\cup}

\newcommand{\proves}{\vdash}
\newcommand{\Trans}{\mathrm{Trans}}

\newcommand{\PRA}{\mathsf{PRA}}
\newcommand{\PRS}{\mathsf{PRS}}

\newcommand{\WKL}{\mathsf{WKL_0}}

\newcommand{\ZFC}{\mathsf{ZFC}}

\newcommand{\KP}{\mathsf{KP}}

\renewcommand{\And}{\wedge}
\newcommand{\Or}{\vee}
\newcommand{\Yields}{\;\Rightarrow\;}

\newcommand{\Implies}{\mathop{\, \rightarrow\,}}
\newcommand{\Iff}{\, \leftrightarrow\,}
\newcommand{\INFER}{\;\Rightarrow\;}
\newcommand{\Truth}{\top}
\newcommand{\Falsehood}{\bot}
\newcommand{\Equivalent}{\;\Leftrightarrow\;}
\newcommand{\Proves}{\vdash}
\renewcommand{\deduce}{\infer}
\newcommand{\OR}{\bigvee}
\newcommand{\AND}{\bigwedge}

\newcommand{\Exists}{\exists}

\renewcommand{\S}{\mathcal S}

\renewcommand{\AA}{\mathbb A}

\newcommand{\DD}{\mathbb D}

\newcommand{\PP}{\mathbb P}

\newcommand{\TT}{\mathbb T}

\newcommand{\iI}{\mathfrak I}

\newcommand{\kK}{\mathfrak K}
\newcommand{\lL}{\mathfrak L}

\newcommand{\Kappa}{\mathrm{K}}

\begin{document}

\maketitle

\begin{abstract}
\small
We investigate the position that foundational theories should be modelled on ordinary computability. In this context, we investigate the metamathematics of $\Sigma$ formulas. We consider theories whose axioms are implications between $\Sigma$ formulas. We show that arbitrarily strong such theories prove their own correctness.
\end{abstract}

\smallskip

\noindent \textbf{Motivation.}
The standard distinction between the metalanguage and the object language in metamathematics entails that we never consider the metamathematics of the mathematical universe as a whole. Instead, we study each structure from an external vantage point: an extension or expansion of that structure, or both. The necessity of the metalanguage/object-language distinction is of course a consequence of Tarski's undefinability theorem. The basic premise of our approach is to restrict the logic of the foundational system to circumvent Tarski's theorem. Our approach follows most clearly the work of Feferman. Our discussion of procedures is very much in the spirit of his Operational Set Theory \cite{Feferman09}. We also adopt the position that classical reasoning is valid for initial segments of the set theoretic universe, but not for the universe as a whole. However, we reject even intuitionistic reasoning for the set theoretic universe as a whole, as it is sufficient for Tarski's theorem. Instead, we consider a formula meaningful just in case it is equivalent to some procedure halting; hence, we term the approach ``positivistic''. In the context of set theory, these are essentially the $ \mathbf \Sigma$ formulas. It is an elementary fact that the truth predicate for $\mathbf \Sigma$ sentences is itself $\mathbf \Sigma$, so we may investigate the metamathematics of this fragment without an external vantage point, and therefore, for the mathematical universe as a whole.

\smallskip

\noindent \textbf{Results.}
We define a postivistic theory to consist of sequents of $\mathbf \Sigma$ formulas. We define a positivistic proof to be a sequence of $\mathbf \Sigma$ formulas, with each formula obtained from the preceding formula according to an axiom, which may be applied deeply (\cref{block: positivistic proof}). We obtain a complete list of logical axioms for positivistic reasoning using a term model construction (\cref{block: completeness}). We show that there is a positivistic theory for pure sets that proves Tarski's semantic axioms for the truth predicate, that proves that the conclusion of any axiom of that theory is true if its assumption is true, and that interprets $\ZFC$; likewise there is such a positivistic theory for natural numbers that interprets $\PRA$ (\cref{block: correctness}). Observing that the powerset operation is not definable by a $\mathbf \Sigma$ formula (\cref{block: powerset}), we propose two completeness axioms for the universe of pure sets, and we show that they are equivalent (\cref{block: equivalence}). These completeness axioms imply that every set is countable. Following Weaver \cite{Weaver}, we define an assertibility predicate for intuitionistic formulas in the language of set theory (\cref{block: assertibility definition}), we show that the assertibility predicate satisfies Tarski's semantic axioms for the positivistic connectives, and in particular, we show that a $\mathbf \Sigma$ sentence is true if and only if it is assertible (\cref{block: assertibility properties}).

\smallskip
\newpage
\small
\noindent \textbf{About these notes.} These notes are not the final draft of a paper; I am posting this draft prematurely for practical reasons. Some details need more attention; some ideas should be expressed more concisely. The foundational requirements of some of the later arguments need to be determined with care. I have included a minimal bibliography, that will need to be substantially extended.

I first started on this line of thought as a student in a course of Burgess. I cite his discussion \cite{Burgess} of the distinction between accepting each derivation of a system and accepting that system as whole. Three authors inadvertently spurred me to finally put pen to paper. First, Weaver's incisive discussion \cite{Weaver} of foundational positions in terms of a choice of objects, constructions and logic, and of sets in terms of surveyability, expresses a very similar approach to the one that I had held vaguely, and am now expressing here. I diverge from Weaver in my position that the choice of logic is essentially forced by the choice of objects and constructions. I regret that I was not able to include anything about surveyablility in this draft. Second, I was struck by Schweber's definition \cite{KnightMontalbanSchweber} of computational reducibility for uncountable objects. The computational properties of a structure do not depend on how that structure is forced to be countable; one simple, perhaps simplistic, explanation for this phenomenon is that the structure simply is countable. Third, I was intrigued by Hamkins's embeddability result \cite{Hamkins}, which suggests to me some hope of explaining the linear ordering of natural set theories by consistency strength.

I am not a logician by training. Some effort went into proving results that are widely known in the logic community. I do not have a good understanding of what in these notes is new. I had obtained the complete list of logical axioms for positivistic proof before I encountered Beklemishev's work on positive deep inference \cite{Beklemishev}. The application of positive deep inference to predicate logic may nevertheless be new. I had obtained the formalization of the naive argument for the validity of bounded proofs before I encountered Pudl\'ak's work on finitistic consistency proofs \cite{Pudlak}. Pudl\'ak's result uses a different measure of size, for both the derivations he considers and the proof itself, and he permits classical logic, in both the derivations he considers and in the proof itself; I do not yet fully understand the relationship between the two results. My impression is that at least the extreme naturality of the proof given here is novel. Takahashi developed a theory of finitary mathematics \cite{Takahashi} that identifies finitary methods with the $\Sigma_1$ definable functions, as does our approach. Takahashi's approach is less general in that it studies a single formal system, which necessarily describes the hereditarily finite sets; furthermore, this system is not positivistic because it includes bounded universal generalization; see \cref{section: model universes}. I expect to encounter other examples of previous work that includes some of the results here, partially or wholly. I am even less familiar with the literature in the philosophy of mathematics.

Our exposition draws heavily on Rathjen's development of $\PRS$ \cite{Rathjen}, the theory of primitive recursive set functions, and Feferman's proof of cut-elimination \cite{Feferman} for an infinitary sequent calculus that includes countable conjunctions and disjunctions. We apply the Rauszer-Sabalski lemma \cite{RauszerSabalski} in our completeness argument.

I thank Lev Beklemishev, Joel David Hamkins, Joost Joosten, Alex Kruckman, Michael Rathjen, and Stephen Simpson, for responding patiently to my invariably elementary questions about proof theory. I thank Nik Weaver for discussion before and after this project. I am grateful to Andrew Marks for his comments on an earlier, messier presentation of these results. I apologize to any people whom I may have left out, in haste.

Figures summarizing various deductive systems appear at the end of these notes. I use a colon ($:$) after each quantifier as a visual separator. The range of the quantifier extends to the next conditional implication symbol ($\Rightarrow$), unless the conditional implication symbol is enclosed by parentheses.

\normalsize

\newpage

\section{summary}\label{section: summary}

The positivistic approach to foundations, introduced here, includes no assumption on what mathematical objects may exist. Rather, it is the position that a mathematical proposition is meaningful if and only if it is verifiable by a mathematical procedure, whatever objects may exist and whatever procedures may be possible. The positivistic approach is then a kind of extension of the positivism of the Vienna Circle to mathematical objects, replacing the human agent with an ideal agent.

In our motivating example of a mathematical universe, the objects are natural numbers and the procedures are recursive partial function presentations. The positivistic position is then that the computably verifiable predicates are the only meaningful ones. This example exhibits many characteristic features of the positivistic approach; in particular one obtains a self-contained universe with a truth predicate given by the universal procedure, and having the nature of a potential, rather than completed totality.

The restriction to computably verifiable predicates is a feature of finitism; we might say that finitism is Turing computability plus positivism. We vary the notion of computation, while retaining the positivism.

In the examples that we consider, the mathematical universe consists of pure sets. We consider model universes $H_\kappa$ of sets of hereditary cardinality less than a regular cardinal $\kappa$. The computable propositions are given by $\Sigma$ formulas (essentially $\Sigma_1$ formulas) in the language of set theory, at times augmented by various function symbols. The axioms are given by implications between $\Sigma$ formulas (essentially $\Pi_2$ sentences). Each axiom expresses the intuition that while the antecedent formula is true of a given tuple of objects, the consequent formula is also true of that tuple of objects. Each axiom permits inference by replacing the antecedent by the consequent, even a substitution instance of the antecedent by the corresponding substitution instance of the consequent, and even when that substitution instance of the antecedent occurs as a subformula. This is called deep inference; it is valid here because we confine negation to atomic formulas.

We obtain a list of logical axioms that is complete for positivistic reasoning. The theory of primitive recursive arithemtic is sufficient to prove the reduction of intuitionistic and classical derivations to positivistic proofs. With the additional axiom of iterated reflection through the ordinals, infinitary intuitionistic derivations can also be reduced to positivistic proofs. Defining the assertiblity of an intuitionistic sentence to be its derivability in infinitary intuitionistic logic, from the axioms of a fixed theory and the true $\Sigma$ sentences with parameters, we find that assertibility, like the truth predicate, respects the meaning of the positivistic logical connectives.

Thus, the positivistic approach does not reject the existence of any object, nor the feasibility of any procedure, nor the validity of any classical deduction, nor even the justifiability of any $\Pi_2$ axiom. The impact of taking the positivistic position is the rejection of axioms on the basis of their logical complexity. Among the axioms of $\mathsf{ZFC}$, the powerset axiom may be rejected on this basis. If we do not assume the feasibility of the powerset operation, limiting ourselves to $\Sigma$ formulas in the langague of set theory, then we cannot express the powerset axiom as a $\Pi_2$ formula.

Furthermore, an inclusive conception of the universe of pure sets suggests that all pure sets are countable. The forcing method potentially produces a bijection between any given infinite set and the set of natural numbers; this bijection should belong to the universe of all possible pure sets. The forcing method is not simply a central tool in our understanding of set theory, it is a special case of the principle of completeness for infinitary logic, itself a natural part of the inclusive conception.

Following the positivistic approach, we are thus led to rejecting the powerset axiom in favor of a universe of pure sets that is complete and self-contained in a number of ways. It is a mathematical universe that contains its own metamathematics; we avoid Tarski's theorem by rejecting the notion that every proposition must have a negation. Instead, G\"odel's theorems and reflections principles come to the foreground.

Any given theory, that is, a class of axioms given by some unary predicate has a correctness principle, which asserts that if the premise of some axiom is true, then its conclusion is also true. Such a correctness axiom proves its own correctness: it can be used to show that if the premise of any axiom is true, then its conclusion is also true. Thus, the resulting theory proves that each of its inferences is correct! Is this in contradiction with G\"odel's second incompeleteness theorem? No, because the induction rule is not admissible for such a theory. We have only obtained, for every natural number $n$, the proof that if the first line of an $n$-line proof is the true atomic sentence $\Truth$, then the second line of the proof is true, then the third line of the proof is true, then ... , then the $n$-th line of that proof is true, so the proof does not prove the false atomic sentence $\Falsehood$. The reader will surely agree that though such a theory does not prove its own consistency, it very nearly does so.

\section{potentialism}\label{section: potentialism}

The positivistic approach may be considered a form of mathematical potentialism. Potentialism is, roughly, the position that the mathematical universe is not a completed totality, and that its objects cannot be taken to all exist simultaneously. Potentialism has its roots in classical inquiry into the nature of the infinite. In the context of set theory, the motivation for potentialism may be summarized as follows: First, the paradoxes of naive set theory suggest that there is something alien about the totality of all sets and the totality of all ordinals. Second, the intuitive framework underlying large cardinal axioms suggests that the universe of sets is not vertically completed. Third, set-theoretic pluralism arising from the forcing method suggests that the universe is not horizontally completed either.

The positivistic approach also presents an account of truth in mathematics. The major obstacle to the naive conception of truth is the lair paradox, formalized as Tarski's undefinability theorem. Tarski's undefinability theorem is the proposition that any system that includes classical first-order logic, a truth predicate, and minimal axioms about finitistic mathematics, is necessarily inconsistent. Consequently, foundational theories typically lack a truth predicate, leading to the necessary metamathematical distinction between the object-language and the metalanguage.

Thus, there is a sense of incompleteness both in the standard account of the mathematical universe, and in the standard account of truth. This incompleteness can be approached using hierarchies or modalities. The motivation of this paper is to describe the mathematical universe as a whole. For this purpose, we give up classical first-order logic in exchange for a universal truth predicate. We must even give up intuitionistic first-order logic, because it is sufficient for the proof of Tarski's undefinability theorem. This paper accepts only $ \mathbf\Sigma$ sentences as meaningful propositions about the mathematical universe.

The true motivation of this paper is toward an account of infinitary data. The interpretation of mathematics in set theory reduces mathematical objects to their information content, with pure sets in the role of datums. This is the only significance of sets in this paper; we look at pure sets and formulas because these constitute the established foundation of infinitary mathematics, and the established subject of metamathematical research. In a research vacuum, I would prefer to study sets of ordinals, i. e., well-ordered sequences of bit values, and procedures on these sequences. Where these two models of infinitary data diverge for choice reasons, I trust the latter.

\section{objects and procedures}\label{section: objects}

The positivistic viewpoint includes no position on what mathematical objects may exist or what mathematical procedures are possible; it is a position on what predicates about these objects and procedures are meaningful. We use the term ``procedure'' in a wide sense, to include any abstract way of obtaining an object from given objects. We include nondeterministic procedures such as Vitali's ``construction'' of a nonmeasurable set of real numbers; in particular, we include the procedure that produces an arbitrary mathematical object.

The reader is invited to take a position on what mathematical objects may exist and what mathematical procedures are possible, out of conviction, or for the sake of argument. Your choice of objects may reasonably be termed your mathematical ontology, and your choice of procedures may, more controversially, be termed your mathematical epistemology. This latter term may be justified by the notion that the phrase ``mathematical procedure'' does not refer to a mechanical method, but to a way of deriving abstract entities. This choice of mathematical ontology and mathematical epistemology determines the mathematical universe. 

Our use of word ``epistemology'' captures the essence of the positivistic viewpoint: propositions are mathematical objects, and truth is a mathematical procedure, in the sense above. Metamathematical objects are mathematical objects, and any way of assigning truth to propositions must be part of the mathematical universe, if indeed the mathematical universe includes all mathematical objects that may exist and all possible mathematical procedures. Crucially, mathematical procedures must themselves be counted among mathematical objects.

We take a proposition to be a nullary procedure, i. e., a procedure from no given objects. We call a proposition true when it produces a specified ``truth'' object, which we may take to be the $0$-tuple. More generally, an $n$-ary predicate is just an $n$-ary procedure; when an $n$-ary predicate produces the truth object from $n$ given objects, we say that it is true of those objects. We assume that each object has a canonical name, i. e., a nullary procedure that produces just that object, so we can express the truth of an $n$-ary predicate of $n$ given objects as a proposition.

We can implement each positive Boolean connective by composing the given propositions with a specific procedure. Conjunction is implement as the binary procedure that produces the truth object when both given objects are the truth object. Disjunction is implemented as the binary procedure that produces the truth object when either given object is the truth object. Falsehood is implemented as the the nullary procedure that produces no objects.

The usual diagonalization argument shows that there is no procedure that produces a negation of the given proposition, i. e., a proposition that produces the truth object just in case the given proposition does not. More generally, there is no procedure that produces an implication of two given propositions. Though we may be willing to infer one proposition from another, the validity of this inference cannot generally be expressed as a proposition, and so cannot be a truth about the mathematical universe.

A theory, vaguely speaking, consists of the inferences one is willing to make; we formalize it essentially as a binary predicate. There is no general notion of truth for theories, since there is no general notion of truth for a single inferences. Furthermore, there is no general notion of falsehood. We would like to say that a false theory is one that infers a false proposition from a true proposition, but we have no falsehood predicate. We are reduced to doubting theories that draw dubious conclusions.

\section{formulas}\label{section: formulas}

To simplify the presentation we break with the very general discussion of the previous section by supposing a universe of pure sets. Many of the results that follow hold in more general settings, but clearly enunciating the necessary assumptions muddles the presentation. The interpretation of the mathematical universe in its class of pure sets is ubiquitous in the study of foundations, and it is useful to us for precisely this reason. We justify this interpretability assumption by remarking that pure sets appear to serve adequately as infinitary datums. Thus, we assume that the objects of the mathematical universe are pure sets.

We model our notion of procedure on ordinary computability. Specifically we imagine a register machine, with each register holding a pure set that varies over the course of the procedure. Appropriately, we label each register with a variable symbol; while we consider finitary logic, it is enough to imagine countably many registers, with one register and one variable symbol for each natural number. Within the universe of pure sets, a natural number is naturally a finite von Neumann ordinal.

We express procedures in first-order logic. Each notion of procedure is given by the cardinal characteristics of its logic, and the functions and predicates in its vocabulary. Our vocabulary will consist of the binary predicate symbols $=$, $\neq$, $\in$, and $\not \in$, and various function symbols, generally definable by a $L_{\omega \omega}(= , \neq, \in, \not\in)$ formula. 
We will primarily focus on the logics $L_{\omega \omega}$, and $L_{\infty \omega}$. In a universe that does not satisfy the axiom of choice, we will write $L_{\Omega \omega}$ to refer to the class of infinitary formulas in which the conjunctions and disjunctions are well-ordered. 

The notation $L_{\omega \omega}(\S)$ refers to formulas in a given vocabulary $\S$. When the vocabulary $\S$ is not specified, it is understood to be the vocabulary of our set theory, usually consisting of just the predicate symbols $=$, $\neq$, $\in$, and $\not \in$, but possibly including the vocabulary of $\PRS$, or even the symbol $\wp$ for the powerset construction . 

A formula is in $\Sigma(L_{\omega \omega}(\S))$ if it is build up from atomic formulas, including the logical symbols $\Truth$ and $\Falsehood$, using binary conjunction, binary disjunction, bounded universal quantification, and unbounded existential quantification. We accept bounded universal quantification, and reject unbounded universal quantification, because we assume that each set is surveyable, but the universe as a whole is not. (Surveyability is a term introduced by Weaver \cite{Weaver} that for our purposes is essentially synonymous to definiteness, as of a totality; I sometimes prefer this term for its procedural intuitions.) We include the negated  predicate symbols $\neq$ and $\not\in$, but exclude the negation connective, because we assume that it is possible to verify that given pure sets are distinct, or that one is not an element of the other, but that it is generally impossible to verify that a given formula is not true. Indeed, our attitude is that it is generally \emph{meaningless} to speak of a given formula not being true.

Similarly, we say that a formula is in $\mathrm I(L_{\omega \omega}(\S))$ if it is built up using conjunction, disjunction, unbounded existential quantification, unbounded universal quantification, and implication; the letter $\mathrm I$ suggests intuitionistic logic, as in Gentzen's $\mathfrak{LI}$, the letter $\mathfrak{I}$ commonly mistaken for $\mathfrak{J}$. We regard the formulas of $\mathrm I(L_{\omega \omega})$ as \textit{prima facie} meaningless.

We would like the meaning of $\mathrm K(L_{\omega \omega}(\S))$ to depend on context; the letter $\mathrm K$ suggests classical logic, as in Gentzen's $\mathfrak{LK}$. In the context of the G\"odel-Gentzen translation, it is natural to say that a formula is in $\mathrm K(L_{\omega \omega}(\S))$ if it is built up using conjunction, unbounded universal quantification, and implication. However, for proof theoretic arguments it is more convenient to represent disjunction, unbounded existential implication, and negation as separate symbols. Furthermore, in the context of Friedman's translation, it is convenient to work with the same class of formulas as in the context of G\"odel-Gentzen translation, but to parse them differently. For definiteness, we will say that a formula is in $\mathrm K(L_{\omega \omega}(\S))$ if it is built up using conjunction, disjunction, unbounded universal quantification, unbounded existential quantification, and negation.

Finally, we write $\mathbf \Sigma$, $\mathbf I$, and $\mathbf K$ in place of $\Sigma$, $\mathrm I$ and $\mathrm K$, to indicate that parameters are permitted. Note that because truth is definable for $\mathbf \Sigma(L_{\infty \omega})$ sentences, the formulas in $\mathbf \Sigma(L_{\infty \omega})$ do not yields a strictly stronger notion of computability than the apparently weaker $\mathbf \Sigma(L_{\omega \omega})$. Note also that within the universe of hereditarily finite sets, $L_{\infty \omega}$ and $L_{\omega \omega}$ essentially coincide, with the latter formally restricted to \emph{binary} conjunction and disjunction.

We write $\phi^v_t$ and $\phi(v/t)$ for the result of substituting the term $t$ for the variable $v$. We write $\phi(v_1, \ldots, v_n)$ for a formula whose free variables are among $v_1, \ldots, v_n$, and we write $\phi(t_1, \ldots, t_n)$ as a shorthand for $\phi(v_1/t_1, \ldots, v_n/t_n)$. We write $\underline a$ for the parametric symbol naming the object $a$, and we write $\phi[a]$ as a shorthand for $\phi(\underline a)$. We sometimes place corners around a string, e. g., $\quote{\forall s \in a\: s \neq a}$, to denote that string, but more often we blur this important distinction, writing $\phi(x/t)$ where we should write $\phi(\quote{x}/t)$, and writing ``the symbol $\pi$'' where we should write ``the symbol $\quote{\pi}$''. We may abbreviate the phrase ``substitution instance'' by the word ``instance''.

\section{positivistic proof}\label{section: proof}

If we intuit a procedure $q$ to have at least the outputs of procedure $p$ on every tuple of inputs, we write $p \Yields q$, and we feel justified in replacing any occurrence of $p$, as a subprocedure in a proposition, by $q$. In the context of a universe of pure sets, we only directly consider procedures given by $\mathbf \Sigma$ formulas, so if we intuit that $\psi(x_1, \ldots, x_n)$ is true whenever $\phi(x_1, \ldots, x_n)$ is true for every tuple of values, we write  $\phi(x_1, \ldots, x_n) \Yields \psi(x_1, \ldots, x_n)$. Such an implication between $\mathbf \Sigma$ formulas is called a \underline{conditional}.

Replacing a subprocedure $\phi$ by $\psi$ amounts to replacing a \emph{subformula} of the form $\phi(t_1, \ldots t_n)$ by a subformula of the form $\psi(t_1, \ldots, t_n)$, where the terms $t_1, \ldots, t_n$ may include bound variables. If the subformula $\phi(t_1, \ldots t_n)$ occurs below a bounded universal quantifier, then we may think of the subprocedure $\phi$ occuring within a kind of $\texttt{for}$ loop. An inference on a subformula, such as this, is termed ``deep inference'' in the literature.

A theory consists of conditionals $\phi(x_1, \ldots, x_n) \Yields \psi(x_1, \ldots, x_n)$, which are called the axioms of the theory. Formally, we distinguish between intentional theories and extensional theories. An \underline{intentional theory} is formula $\tau(x)$, which is satisfied by its axioms, whereas an \underline{extensional theory} is just the set $T$ of its axioms. We say that a structure models an extension theory $T$ just in case that structure models the universal closure $\overline{\phi(x_1, \ldots, x_n) \Yields \psi(x_1, \ldots, x_n)}$ of every element of $T$, in the ordinary sense. The notion of ``a model'' is more subtle for intentional theories. We say that a structure models an intentional theory $\tau$ just in case it models the universal closure of the correctness principle of $\tau$, which we define in \cref{section: truth}.

\begin{definition}\label{block: positivistic proof}
Let $\phi(v_1, \ldots, v_n) \Yields \psi(v_1, \ldots, v_n)$ be a conditional, that is, an implication between $\mathbf \Sigma(L_{\omega \omega})$ formulas. An \underline{application} of $\phi \Yields \psi$, is a pair of $\mathbf \Sigma(L_{\omega \omega})$ formulas $(\chi_0, \chi_1)$ such that $\chi_1$ is obtained from $\chi_0$ by replacing a subformula of $\chi_0$ that is of the form $\phi(t_1, \ldots, t_n)$ by $\psi(t_1, \ldots, t_n)$.

An \underline{extensional theory} $T$ is a set of conditionals, and its elements are called its axioms. An \underline{intentional theory} $\tau$ is a $\mathbf \Sigma(L_{\omega \omega})$ formula with a single free variable. An \underline{axiom} of $\tau$ is a conditional $\phi \Yields \psi$ that satisfies this formula.
A finitary positivistic \underline{proof} of $\chi \Yields \chi'$ in $\tau$ is a finite sequence of $\mathbf \Sigma(L_{\omega \omega})$ formulas $(\chi_0, \ldots , \chi_n)$ such that $\chi_0 = \chi$, $\chi_n = \chi'$, and each step $(\chi_i, \chi_{i+1})$ is an application of some axiom.
\end{definition}

Is the appearance of free variables in the axioms and formulas of a positivistic proof compatible with the positivistic rejection of universal claims? This question is addressed in \cref{section: remarks}. Briefly, the free variables that appear in the formulas that make up a positivistic proof are assumed to have fixed values for the entire course of the proof; they may be treated as syntactic variables. The free variables that appear in axioms do have a universal sense; however, axioms express inference rules rather than propositions. The absolute justifiability of a given inference is not a positivistically meaningful proposition.

\section{the logical axioms}\label{section: completeness}

Our informal remarks have been in or about some hypothetical universe. For concreteness, the following more formal remarks will be about \emph{model universes} under standard foundations, i. e., $\ZFC$, possibly augmented by large cardinal axioms. We begin by obtaining a complete list of logical axioms for positivisitic proof. Thus, we consider all models of any vocabulary containing distinguished binary predicates $=$, $\neq$, $\in$, and $\not \in$, such that $=$ and $\neq$ denote complementary relations, $\in$ and $\not \in$ denote complementary relations, and $=$ denotes genuine equality.

We initially obtain our list of logical axioms for positivistic proof by constructing a completeness argument in the standard way. This argument may be viewed as taking place in $\ZFC$, where we are considering model universes. 

First, we obtain a sequent calculus, whose sequents are conditionals $\phi \Yields \psi$. Fix an extensional theory $T$, and a conditional $\chi \Yields \chi'$ that cannot be proved from $T$ using rules of inference to be added retroactively. Define a preorder on all $\Sigma$ formulas with $\phi \leq \psi$ just in case $\phi \Yields \psi$ is provable from $T$. We obtain a term model of $T$ that satisfies $\chi$ but not $\chi'$ from a prime filter on the preorder that contains $\chi$, but not $\chi'$. This argument is entirely standard apart from the difficulty that our preorder is not complemented, because we do not have negation. To ensure the existence of a prime filter that respects the infinitary joins and meets corresponding to quantified formulas, we add distributive axioms that are unnecessary in the classical case. We arrive at the conditional calculus in figure 1.

\begin{lemma}[$\ZFC$]
Let $\S$ be a countable vocabulary consisting of function symbols and predicate symbols of any finite arity, including $=$, $\neq$, $\in$ and $\not \in$. Let $T$ be a set of $ \Sigma(L_{\omega \omega}(\S))$ conditionals. Let $\chi \Yields \chi'$ be any $\Sigma(L_{\omega \omega}(\S))$ conditional. Then, either $\chi \Yields \chi'$ is derivable from $T$ using the system given in figure 1, or there is a model $M$ of $T$ where $\chi \Yields \chi'$ is false for some tuple of elements.
\end{lemma}

\begin{proof}
Define $\phi \leq \psi$ whenever $\phi \Yields \psi$ is derivable from $T$. Rules (-2) and (-1) imply that $\leq$ is a preorder on $\Sigma(L_{\omega \omega}(\S))$. Rules (1)-(9) imply that, modulo equivalence, $( \Sigma(L_{\omega \omega}(\S)),\leq)$ is a bounded distributive lattice.

As usual for a term model construction, we distinguish between free variables, $a$, $b$, $c$,..., and bound variables, $x$, $y$, $z$,.... Rules (10), (11), and (0), imply that $\exists x\: \phi$ is the join of $\{\phi^x_t\: t\text{ is a term}\}$. Rules (13), (14), and (0), imply that $\forall x\in s\: \phi$ is the meet of $\{t \not \in s \Or \phi^x_t\: t \text{ is a term}\}$. Rules (12) and (15) imply that these infinitary meets and joins are distributive.

Assume $\chi \not \leq \chi'$. The Rauszer-Sabalski lemma\cite{RauszerSabalski} guarantees the existence of prime filter that preserves the meets and joins described above. As usual, we define the universe of $M$ to be the set of all terms in the vocabulary $\S$ using only the free variables, and we define terms $t_1, \ldots, t_n$ to satisfy the formula $\phi(a_1, \ldots, a_n)$ just in case $\phi(a_1/t_1, \ldots, a_n/t_n)$ is the prime filter. Our choice of filter guarantees that this definition satisfies Tarski's definition of truth. Rule (0) implies that $M$ is a model of $T$, while the predicate symbols $=$, $\neq$, $\in$, and $\not \in$ are treated as being entirely nonlogical.

Rules (16)-(19) imply that $=$ and $\neq$ denote complementary relations on $M$, and likewise for $\in$ and $\not \in$. Rules (20)-(22) imply that $=$ denotes an equivalence relation on $M$. Finally, rule (23) implies that our definition of satisfaction respects this equivalence relation. We conclude that $M/\!\!=$ is a model of $T$.
\end{proof}

We add logical axioms as they become convenient to eliminate the nonaxiomatic rules of inference given in figure 1.

\begin{theorem}[$\ZFC$]
Let $\S$ be a countable vocabulary consisting of function symbols and predicate symbols of any finite arity. Let $T$ be a set of $ \Sigma(L_{\omega \omega}(\S))$. Let $\chi \Yields \chi'$ be any $\Sigma(L_{\omega \omega}(\S))$ conditional. Then, either $\chi \Yields \chi'$ has a finitary positivistic proof using the axioms of $T$ and the logical axioms given in figure 2, or there is a model $M$ of $T$ where $\chi \Yields \chi'$ is false for some tuple of elements.
\end{theorem}

\begin{proof}
Proof is by induction on the length of the derivation produced by the preceding lemma. If the terminal conditional is obtained by the axiomatic rule (-2), it has a proof in zero steps. If the terminal conditional is obtained by the rule (-1), then it has a proof obtained by concatenating the proofs of the input conditionals.

If the terminal conditional is obtained by the rule (0), then it has a proof obtained by substituting $t$ for $a$ in every formula of the proof of the input conditional. Indeed, if $\chi_{i} \Yields \chi_{i+1}$ is an application of some axiom $\phi \Yields \psi$, then $(\chi_{i})^a_t \Yields (\chi_{i+1})^a_t$ is also an application of this axiom. Specifically, if a subformula of the form $\phi^{\overline b}_{\overline s}$ is replaced by $\psi^{\overline b}_{\overline s}$ to infer $\chi_{i+1}$ from $\chi_{i}$, then a subformula of the form $(\phi^{\overline b}_{\overline s})^a_t$ is replaced by $(\psi^{\overline b}_{ \overline s})^a_t$ to infer $(\chi_{i+1})^a_t$ from $(\chi_i)^a_t$. Note that the terms in $\overline s$ may have bound variables, and that $a$ may be among the variables of $\overline b$, but the term $t$ has no bound variables, and the double substitutions in $(\psi^{\overline b}_{ \overline s})^a_t$ may be resolved as single substitutions.

If the terminal conditional is obtained using an axiomatic rule (1)-(23), then it has a proof obtained by applying the corresponding axiom. If the terminal conditional is obtained by the rule (5), then it has a proof obtained by applying axiom (5), and the running the proofs of the input conditionals below the new conjunction symbol. Similarly, if the terminal conditional is obtained by rule (8), then it has a proof obtained by running the proofs of the input conditionals below the highest disjunction symbol, and then applying axiom (8).

If the terminal conditional is obtained by the rule (11), then it has a proof obtained essentially by running the proof of the input conditional below the highest existential quantifier, and then applying axiom (11). Fix a proof $\chi_0 \Yields \cdots \Yields \chi_n$ of the input conditional $\phi^y_a \Yields \psi$. Let $x$ be a variable that does not occur in this proof; then the sequence 
$$\exists x\: (\chi_0)^a_{x}\Yields  \exists x\: (\chi_1)^a_{x}\Yields \cdots\Yields \exists x\: (\chi_n)^a_{x}$$
is also a proof. The same observation may fail for $y$ in place of $x$, because $a$ may be in the scope of a quantifier on $y$ in some $\chi_i$. We now observe that $\exists x\: (\chi_n)^a_{x} \Yields \psi$ is an application of axiom (11), since neither $x$ nor $a$ occurs in $\chi_n = \psi$. We also have a proof of $\exists y\: \phi \Yields \exists x\: (\chi_0)^a_{x}$:
$$\exists y\: \phi = \exists y\: (\phi^y_a)^a_y = \exists y\: (\chi_0)^a_y =
\exists y\: ((\chi_0)^a_x)^x_y \Yields \exists y\: \exists x\:(\chi_0)^a_x \Yields \exists x\:(\chi_0)^a_x$$
We apply axiom (10) in the first inference; we apply axiom (11) in the second inference, since $y$ is not free in $\exists x\:(\chi_0)^a_x$.

If the terminal conditional is obtained by the axiomatic rule (13), then it has a proof using axiom (13) and other axioms.  If the terminal conditional is obtained by the rule (14), then it has a proof obtained essentially by applying axiom (14), and then running the proof of the input conditional below the highest conjunction. Fix a proof $\chi_0 \Yields \cdots \Yields \chi_n$ of $\psi \Yields a \not \in s \Or \phi^y_a$. Let $x$ be a variable that does not occur in this proof; then the sequence
$$\forall x \in s\: (\chi_0)^a_x \And x \in s \Yields \forall x \in s\: (\chi_2)^a_x \And x \in s\Yields \cdots \Yields \forall x \in s\: (\chi_n)^a_x \And x \in s$$
is also a proof. Since neither $x$ nor $a$ occur in $ \chi_0  = \psi$, the conditional $$\psi \Yields \forall x \in s\: (\chi_0)^a_x \And x \in s$$ is an application of axiom (14). We also have a proof of $\forall x \in s\: (\chi_n)^a_x \And x \in s \Yields \forall y \in s\: \phi$:
$$\forall x \in s\: (\chi_n)^a_x \And x \in s \Yields \forall y \in s\: y \in s \And [\forall x \in s\: (\chi_n)^a_x \And x \in s] \Yields \forall y \in s\: (\chi_n)^a_y \And y \in s$$
$$\Yields  \forall y \in s\: ((y \not \in s \Or \phi) \And y \in s) \Yields \forall y \in s\: \phi$$
\end{proof}

We modify the list of logical axioms in figure 2 to remove the distinction between free and bound variables.

\begin{corollary}[$\ZFC$]\label{block: completeness}
The above theorem holds also for the list of logical axioms in figure 3, where we make no distinction between free and bound variables, i. e., each variable is permitted to appear in a free or a bound position.
\end{corollary}

\begin{proof}
In the application of an axiom $\phi \Yields \psi$, we replace a substitution instance of $\phi$ by a substitution instance of $\psi$, so in figure 2 we may replace schematic term variables by schematic variable variables, and in fact, by simple variables. In our proof of completeness we have implicitly assumed a bijective correspondence between the free variables and the bound variables, as formally variables may occur as both free and bound in even a single axiom of the given theory $T$. In identifying corresponding pairs of variables, we must place some restriction on axioms (11), (12), (14), (15), and (23) to preserve their validity. It remains to show that we may find a proof of $\chi \Yields \chi'$ even with these restrictions. Since the restrictions are relevant only when the same variable appears both free and bound in the same axiom, it is enough to find proofs of $\chi \Equivalent \grave{\chi}$ whenever $\grave \chi$ obtained from $\chi$ by replacing each bound variable with a new variable, and likewise for $\chi'$. This follows by induction on complexity. Let $\grave x$ be a new variable.

$$\exists x\: \phi \Yields \exists x\: \exists \grave x\: \phi^x_{\grave x} \Yields \exists \grave x\: \phi^x_{\grave x} \Yields \exists \grave x \: \exists x\: (\phi^x_{\grave x})^{\grave x}_x \Yields \exists \grave x\: \exists x\: \phi  \Yields \exists  x\: \phi $$

\begin{align*}\forall x \in z\: \phi \Yields &\forall \grave x \in z\: (\forall x \in z\: \phi) \And \grave x \in z
\Yields \forall \grave x \in z\: \phi^x_{\grave x} \\ \Yields &\forall x \in z\: (\forall \grave x \in z\:\phi^x_{\grave x}) \And x \in z \Yields  \forall x \in z\: (\phi^x_{\grave x})^{\grave x}_x \Yields \forall x \in z\: \phi
\end{align*}
\end{proof}

The completeness theorem can also be proved by reducing classical logic to positivistic logic via the cut-elimination theorem. This approach has the advantage that it justifies the use of classical logic in a finitary universe, where the completeness theorem itself does not hold. The proof does not appear in these notes, but we include a description of the system $\lL\Sigma_{\omega \omega}(\tau)$, which forms an intermediate step in the argument.

\section{the base theory}\label{section: base theory}

The positivistic approach rejects axioms exclusively on the basis of their logical complexity. Nevertheless, we will refer to a ``base theory'' that is intended to include uncontroversial axioms about basic set-theoretic constructions. We shy away from fully specifying the base theory, to a lesser degree out of principle, and to a greater degree out of practical necessity. The principle just mentioned is that the positivistic approach is not tied to any particular theory. The practical consideration is that it is easy to misjudge the foundational assumptions necessary for a given argument, particularly the quantity of induction. I have not had a chance to comb over the arguments to be certain of their foundational assumptions, but I hope that the reader can agree that the arguments are indeed basic.

At a minimum, our base theory includes $\PRS$, the theory of primitive recursive set functions, introduced by Rathjen \cite{Rathjen} as an adaptation of $\PRA$ to the set-theoretic setting. Primitive recursive set functions on pure sets are a natural generalization of primitive recursive functions on the natural numbers, first extensively studied by Jensen and Karp. I like a variant of this formalization that takes pairing, union, projection, and $\Delta_0$ separation as initial constructions, and introduces new constructions using composition, image, and recursion; these axioms are given in figure 7.

Generalizing the theorems for $\PRA$, Rathjen showed that a $\Sigma_1$ formula in the language of set theory defines a primitive recursive set function if and only if $\mathsf{KP}^-$ with $\Sigma_1$ foundation proves that the formula is a function class; $\mathsf{KP}^-$ is Kripke-Platek set theory with the foundation schema replaced by the much weaker set foundation axiom, i. e. regularity axiom. Rathjen also showed that a $\Pi_2$ sentence in the language of set theory is a theorem of $\PRS$ if and only if it is a theorem of $\KP^-$ with $\Sigma_1$ foundation. Instances of the $\Sigma_1$ foundation schema are not $\Pi_2$, so the latter theory is not positivistic. However, it is possible to present $\PRS$ as a positivistic theory in the language of set theory by coding the primitive recursive set functions by $\Sigma$ formulas, and adding axioms expressing their totality. This fact is significant to us, because it shows that primitive recursive set functions are procedures as long as the membership relation is given by a procedure.

Our base theory includes two natural axioms that do not follow from $\PRS$. The first is $\Delta_0$ collection:
$$\forall x \in X\: \exists y \: \phi \Yields \exists Y\: \forall x \in X \: \exists y \in Y\: \phi.$$
The second is the well-ordering principle: every set is equinumerous to a von Neumann ordinal. The $\Delta_0$ collection schema may be justified on the basis that the set $X$ is an actual totality, which may be surveyed, with witnesses for $\phi$ selected. The well-ordering principle may be justified on the basis that while we talk about pure sets, the actual objects of our universe are well-ordered sequences of bit-values, i. e., sets of ordinals; thus, not only can a set be well-ordered, but on the ``machine level'' a set is represented by a well-ordered list of its elements. We will rarely consider universes that do not satisfy the well-ordering principle.

With the addition of the $\Delta_0$ collection schema, we see that our base theory includes $\KP^-$. It certainly does not include all of $\KP$, because most instances of the full foundation schema are not expressible by a positivistic conditional. However, the base theory is sufficient for several of the most basic consequences of $\KP$. Specifically, the base theory is sufficient for the proofs of the $\Sigma$ reflection schema, the $\Sigma$ collection schema, and the $\Delta$ separation rule, which are given in Barwise \cite[s. I.4]{Barwise}. The $\Sigma$ reflection schema expresses that each $\Sigma$ formula $\phi$ is equivalent to $\exists x\: \phi^{(x)}$, where $\phi^{(x)}$ is obtained from $\phi$ by replacing every \emph{unbounded} existential quantifier $\exists v$ by $\exists v \in x$. The $\Sigma$ collection schema is just like the $\Delta_0$ collection schema given above, but for all $\Sigma$ formulas $\phi$. The $\Delta$ separation rule is the rule that if we derive that $\Sigma$ formulas $\phi_\Truth$ and $\phi_\Falsehood$ are complementary, then we can derive that each set has a subset of elements satisfying $\phi_\Truth$:
$$\infer{\Truth\Yields \exists y\:(\forall z \in y\: z \in x) \And (\forall z \in x\: \phi_\Falsehood \Or z \in y) \And(\forall z \in x\: \phi_\Truth \Or z \not \in y)}{\Truth \Yields \phi_\Truth \Or \phi_\Falsehood & \phi_\Truth \And  \phi_\Falsehood \Yields \Falsehood}$$
The $\Delta$ separation principle as stated in Barwise section I.4 \cite{Barwise} is not expressible by a positivistic conditional, but its proof shows that the the above rule is admissible in every extension of the base theory. The $\Delta$ separation rule allows induction for $\Delta$ formulas: the following rule is admissible for every extension of the base theory:
$$\infer{\exists x\: \phi_\Truth \Yields \exists x\: \phi_\Truth \And \forall y \in x \: (\phi_\Falsehood)^x_y}{\Truth \Yields \phi_\Truth \Or \phi_\Falsehood & \phi_\Truth \And  \phi_\Falsehood \Yields \Falsehood}$$

The axiom of infinity, expressing the existence of a limit ordinal, is not formally an axiom
of our base theory, because we wish to include the universe of hereditarily finite sets $H_\omega$. In a universe that does contain infinite sets, it is natural to investigate infinitary logic, on the intuition that it is appropriate to the ideal agent for that universe. Consequently, cut-elimination for infinitary logic will be our main tool in a number of the arguments that follow. Feferman obtained a proof-theoretic proof of cut-elimination for an infinitary system that permits countable conjunctions and disjunctions. The proof appears highly constructive, and I expect it to go through in the extension of $\PRS$ just described, with only minor modifications. It follows a typical framework of a syntactic proof of cut-elimination, by indunction on cut-rank, with a subinduction on derivation height, and with subsubinduction on the surreal sum of heights of subdervations, when the cut-rule is encountered. In \cref{section: universe}, we will apply cut-elimination for a variant of Feferman's system that includes more initial sequents; I believe this variant requires only minor additions to Feferman's argument. In \cref{section: assertibility} we will apply cut-elimination for the intuitionistic version of Feferman's system. Feferman explicitly writes that his development does not ``go into'' non-classical fragments; however, I believe that Feferman's argument can be adapted to intuitionistic logic in the usual starightforward way.

In the unexpected situation that Feferman's argument does not go through in the base theory we have described, we will extend $\PRS$ still further, to accommodate it. This extension should not significantly undermine the credibility of our base theory, because Feferman's procedure is certainly intuitively possible. At the very least, Feferman's argument establishes the consistency of cut-elimination with our base theory; the universe $H_{\omega_1}$ of hereditarily countable sets is our canonical model for the base theory with the axiom of infinity. Thus, countable cut-elimination is effectively an axiom of our base theory. There are cut-elimination results for systems that permit arbitrary infinitary conjunctions and disjunctions, such as those in Takeuti's book. However, we will work with variants of Feferman's system, and we will generally apply cut-elimination for these variants under the additional assumption that every set is countable.

It is natural to extend our base theory with reflection principles. The validity principle for the base theory is the standard example: it expresses that the conclusion of any proof whose assumptions are true is itself true. A set-theoretic variant expresses that each set is an element of some transitive model of the base theory. Such a transitive model is automatically a model of $\mathsf{KP}$. In fact, we obtain a conservation result for $\mathsf{KP}$, since every $\Pi_2$ formula in the language of set theory that is true in each of these transitive models applies to the universe as a whole.

If we stretch our relaxed attitude toward axioms to its extreme, we might accept all the conditionals modelled by some model universe as axioms, and investigate the classical first-order sentences that follow from these axioms.

\section{three model universes}\label{section: model universes}

We examine three model universes, models of the base theory: the levels $H_\kappa$ for $\kappa = \omega$, for $\kappa = \omega_1$, and for $\kappa$ an inaccessible cardinal. We might identify these models universes with finitist, predicativist, and realist foundations.

\subsection{$H_\omega$}

The theory that we obtain by adding the axiom that each set if finite, i. e., equinumerous to a natural number, is equivalent to $\PRA$, the theory of primitive recursive arithmetic, but this takes some work to prove. The Ackermann coding is a simple bijection between hereditary sets and natural numbers that is a primitive recursive set function. We may formulate $\PRA$ itself as positivistic theory; the axioms of postivistic logic make no assumptions about the binary relation $\in$, so we may simply replace it with $<$. The axioms of $\PRA$ may be naturally formulated using conditionals; for example we may render $\Delta_0$ induction as
$$\phi(0, \overline z) \And \forall x < y\: \tilde\phi(x, \overline z) \Or \phi(x+1, \overline z) \Yields \forall x < y\: \phi(x),$$
for each $\Delta_0$ formula $\phi$, with the negation $\tilde \phi$ obtained by switching dual connectives and predicate symbols. The Parsons-Friedman theorem \cite[s. IX.3]{Simpson} implies that every conditional provable in $\WKL$ is positivistically provable from the axioms of $\PRA$.

The original formulation of $\PRA$ is essentially positivistic, in the sense that while it is not prima facie positivistic, it can be easily adjusted to be positivistic. The reason it is not, strictly speaking, positivistic is that the free variables in a deduction of $\PRA$ cannot be understood as having the same values throughout the proof. It is clear that the induction variable in the premise of the induction rule must be read as being implicitly universally quantified in order to justify the application of the induction rule. Thus, a deduction of $\PRA$ consists of $\Pi$ sentences, rather than $\Sigma$ sentences. However, each deduction of $\PRA$ can be read as a positivistic proof if it is read backwards: if the terminal formula fails, then one of the predecessors used to derive it fails, and so on, until the set of possibly failing formulas is found to be empty, i. e., a contradiction is derived.

We remark that Takahashi \cite{Takahashi} identified $\Sigma$ sentences in the language of set theory as the meaningful sentences in the context of finitistic foundations based on hereditarily finite sets. However, Takahashi's deductive system is not positivistic, again because its free variables do not maintain fixed values over the course of deduction. Specifically, a formula of the form $\forall x \in y\: \phi$ can be derived from a deduction that assumes $x \in y$ and proves $\phi$. This bounded variant of the universal generalization rule is a covert appeal to the validity of the system. While these appeals to validity are ultimately innocuous, it is surely a basic principle of modern mathematical logic that no system should assume its own validity.

\subsection{$H_{\omega_1}$} We obtain a basic theory of hereditarily countable sets by adding the axiom of infinity, concretely the existence of a least nonzero ordinal that is not a successor, and the axiom that every set set equinumerous to this ordinal, or to a smaller ordinal. The position that every set is countable is well established; I suppose it is associated most strongly with predicativism, and related denials of the powerset axiom. I am drawn to the conclusion that indeed, if there are infinite sets, then they are all countable, but I am not persuaded by predicativist appeals, possibly because I do not understand them.

My intuition of the positivistic approach is that, though it is conservative with regard to semantics, it is liberal with regard to ontology. Whatever the objects in the universe, they are taken to exist in themselves, rather than be generated by a mental process. Furthermore, although the positivistic approach has the formal features of potentialism, I do not imagine the emergence of objects that seems to underly potentialist intuitions. When I imagine a mathematical universe, I imagine its objects as existing fully; the meaninglessness of a universal proposition does not arise from a dynamic growth of the universe, but rather from an inability to survey all of its objects. Ultimately, this departure from the narrative of potentialism may be entirely aesthetic, i. e., of no practical consequence.

I will argue that the full universe of sets looks very much like $H_{\omega_1}$, but the reader is encouraged to keep a mental distinction between $H_{\omega_1}$, and the full universe of sets. The model universe $H_{\omega_1}$ is an initial segment of the ``ambient'' transitive model of $\ZFC$ in which we have been investigating the notion of a model universe. The position that every set is countable, for which I will argue, is not an exclusive position that rejects the many uncountable sets that are used in ordinary mathematics; rather, it is an inclusive position that accepts the existence of a bijection between the set of natural numbers and its powerset in this transitive model.

As a model universe, $H_{\omega_1}$ has a number of attractive features. It is the only model universe of the form $H_{\kappa}$ that satisfies the two completeness principles given in \cref{section: universe}; infinitary logic is well-behaved, and forcing for transitive models is a provable principle via Rasiowa-Sikorski lemma. Furthermore, by the L\'evy absoluteness principle, every $\Sigma(L_{\omega \omega}(= , \neq, \in , \not\in))$ conditional valid in the ambient model of $\ZFC$ is also valid in $H_{\omega_1}$, so within the ambient model, we cannot express that the ambient model contains uncountable sets using $\Sigma(L_{\omega \omega}(= , \neq, \in , \not\in))$ conditionals.

\subsection{$H_\kappa$}

If $\kappa$ is an inaccessible cardinal, then $H_\kappa$ is a model of $\ZFC$. However, the ideal agent for $H_\kappa$ may not accept the validity of $\ZFC$ in that universe, because a number of its axioms have high logical complexity. It is of course possible to formulate any theory as a $\Pi_2$ theory, or even as a $\Pi_1$ theory by adding nonlogical symbols, essentially Skolem functions, with the implicit claim that these nonlogical symbols denote possible procedures. However, if we take our notion of procedure to be the one given by the language of set theory, i. e., $\mathbf \Sigma(L_{\kappa \omega}(= , \neq, \in , \not \in))$, then there is no powerset procedure in $H_\kappa$:

\begin{proposition}[$\ZFC$]\label{block: powerset} Let $\kappa$ be a strongly inaccessible cardinal. The powerset operation is not definable in $H_\kappa$ by a $\mathbf \Sigma(L_{\kappa \omega}(=, \neq, \in , \not \in))$ formula.
\end{proposition}

\begin{proof}
Suppose that there is a $\mathbf \Sigma(L_{\kappa \omega}(=, \neq, \in , \not \in))$ formula $\phi(x,y)$ such that $H_\kappa \models \phi[a,b]$ iff $b$ is the powerset of $a$, for all $a, b \in H_\kappa$. The formula $\phi$ is in $H_{\lambda^+}$ for some regular $\lambda < \kappa$. The transitive set $H_\kappa = V_\kappa$ satisfies the powerset axiom, so certainly $H_\kappa \models \exists y\: \phi(\underline \lambda,y)$. Truth for $\mathbf \Sigma(L_{\kappa \omega}(=, \neq, \in , \not \in))$ formulas is definable by a $\Sigma(L_{\omega\omega}(=, \neq, \in , \not \in))$ formula, so $H_\kappa \models \TT[ \exists y\: \phi(\underline \lambda,y)]$. Since $H_{\lambda^+}$ reflects $\Sigma$ formulas, we have that $H_{\lambda^+} \models \TT[ \exists y\: \phi(\underline \lambda,y)]$, so $H_{\lambda^+} \models \exists y\: \phi(\underline \lambda, y)$. We conclude that there is an element $p \in H_{\lambda^+}$ such that $H_\lambda \models \phi[\lambda,p]$, and therefore $H_\kappa \models \phi[\lambda,p]$. By assumption, $p$ is the powerset of $\lambda$, contradicting that $p$ is an element of $H_{\lambda^+}$.
\end{proof}

If we view the powerset operation as given by a procedure, we may add it to our vocabulary, together with other such operations. Our general method for adding procedures without adding objects is to fix a coding of elements outside the model by elements of the model; we then expand the vocabulary of that model by a truth predicate, interpreted as composing the genuine truth predicate with that coding. In this way, we may include the metamathematics of a model of $\ZFC$ into that model of $\ZFC$. Note that we are not simply adding a truth predicate in the usual way, because our truth predicate may be applied to formulas that contain that truth predicate, but at the expense restricting our reasoning to positivistic logic. If we code $\mathbf K(L_{\omega \omega}(= , \neq, \in ,\not \in))$ sentences with parameters in the model as $\mathbf \Sigma (L_{\omega \omega}(= , \neq, \in ,\not \in))$ sentences, by relativizing them to that model, then the added truth predicate will apply to these sentences also. Thus, we may retain classical logic for all the usual sentences of the language of set theory, while expanding the model with a self-applicable truth predicate.

\section{truth, correctness, and validity}\label{section: truth}

It is a basic fact that the truth of $\mathbf \Sigma$ sentences is definable by a $\Sigma(L_{\omega\omega})$ formula; this observation applies to $\mathbf \Sigma(L_{\infty \omega})$ sentences just as it does to $\mathbf \Sigma(L_{\omega \omega})$ sentences. Thus, there is a truth predicate in any model universe. For simplicity, we formalize the theory of truth for a vocabulary that contains no function symbols. We may generalize this development readily by adjusting our definition of verification to include a computation of values for the terms appearing in the given formula.

We distinguish between various notions of rightness in the following way. A $\mathbf \Sigma(L_{\infty\omega})$ sentence it \underline{true} if and only if it has a verification in the sense below. As we have emphasized, truth is an internal notion that can be formalized in a given mathematical universe; in contrast, the remaining notions of rightness are expressed by contionals, which may be proved, or taken as axioms, or verified \emph{in a model universe}, but which cannot be in themselves true. The \underline{correctness} of an intentional theory $\tau$ is the reflection principle ``if $\phi_0$ is true and $\phi_0 \Yields \phi_1$ is an application of an axiom of $\tau$, then $\phi_1$ is true''. By contrast, the \underline{validity} of an intentional theory $\tau$ is the reflection principle ``if $\chi_0$ is true and $\chi_0 \Yields \chi_1$ is provable in $\tau$, the $\chi_1$ is true''. The correctness of $\tau$ affirms the axioms of $\tau$, whereas the validity of $\tau$ affirms theorems of $\tau$. When the induction rule is not admissible, the former does not prove the later. Our use of the word ``validity'' is faithful its philosophical meaning: if the assumption of a proof is true then its conclusion is also true; however, it clashes with the standard metamathematical meaning. We will use the phrase ``logical validity'' for the later meaning.

\begin{definition}[base theory]
We define a \underline{verification} $V$ to be a set of $\mathbf\Sigma(L_{\omega \omega})$ sentences that is closed in the following sense:
\begin{enumerate}
\item If $\phi \And \psi$ is in $V$, then $\phi$ and $\psi$ are in $V$.
\item If $\phi \Or \psi$ are in $V$, then $\phi$ is in $V$ or $\psi$ is in $V$.
\item If $\exists x\: \phi(x)$ is in $V$, then a sentence of the form $\phi(\underline a)$ is in $V$.
\item If $\forall x \in \underline b\: \phi(x)$ is in $V$, then for each element $a \in b$, the formula $\phi(\underline a)$ is in $V$.
\item If $\underline a \in \underline b$ is in $V$, then $a \in b$. (Likewise for the other predicate symbols.)
\end{enumerate}
Let $\phi$ be a $\mathbf\Sigma(L_{\omega \omega})$ sentence. We define a \underline{verification of $\phi$} to be a verification that contains $\phi$. We define $\phi$ to be \underline{true} iff it has a verification; we write $\TT(\phi)$.
\end{definition}

There is a straightforward generalization of this definition that includes sentences with function symbols from a finite vocabulary: a verification should include equalities that compute the values of terms. It is more difficult to generalize this definition to an infinite vocabulary, due to the parenthetical remark in condition (5); any definition uses only a finite subset of the vocabulary. In the case of $\PRS$, the algorithms implicit in the primitive recursive function symbols may be unravelled by an appropriate definition of verification. Alternatively, an equation in $\PRS$ may be verified by verifying the corresponding $\mathbf \Sigma$ sentence in the language of set theory.

Our base theory is sufficient for verifying the standard semantic properties of the truth predicate:

\begin{proposition}[base theory]
For all $\mathbf \Sigma(L_{\Omega \omega})$ sentences $\phi$ and $\psi$,
\begin{enumerate}
\item $\TT(\phi \And \psi) \Equivalent \TT(\phi) \And \TT(\psi)$
\item $\TT(\phi \Or \psi) \Equivalent \TT(\phi) \Or \TT(\psi)$
\item $\TT(\exists v\: \psi(v)) \Equivalent \exists a\: \TT(\psi(\underline a))$
\item $\TT(\forall v \in \underline b\: \phi(v)) \Equivalent \forall a \in b\: \TT(\phi(\underline a))$
\item $\TT(\underline a \in \underline b) \Equivalent a \in b$ (Likewise for the other predicate symbols.)
\end{enumerate}
\end{proposition}

\begin{proof}
Immediate from the definition.
\end{proof}

\begin{lemma}[base theory]
Let $\chi_0$ and $\chi_1$ be $\mathbf \Sigma(L_{\omega \omega})$ sentences. Assume $\chi_0 \Yields \chi_1$ is an instance of a logical axiom, and $\chi_0$ is true. Conclude $\chi_1$ is true.
\end{lemma}

\begin{proof}
There is a separate argument for each logical axiom schema. As an example, we suppose that $\chi_0 \Yields \chi_1$ is an instance of the schema $x \in z \And \forall x \in z \: \phi \Yields \phi$. Thus, $\chi_0 \Yields \chi_1$ is equal to $\underline a \in \underline c \And \forall x \in \underline c \: \phi \Yields \phi^x_{\underline a}$ for some $\mathbf \Sigma(L_{\omega \omega})$ formula $\phi$ with a single free variable $x$, and objects $a$ and $c$. Let $V$ be a verification for $\underline a \in \underline c \And \forall x \in \underline c \: \phi$. The set $V$ contains $\underline a \in \underline c$, so $a \in c$. The set $V$ also contains $\forall v \in \underline c\: \phi$, so it also contains $\phi^x_{\underline b}$ for all $b \in c$. Thus, $V$ contains $\phi^x_{\underline a}$, so $V$ is a verification for $\phi^x_{\underline a}$.

As another example, we suppose that $\chi_0 \Yields \chi_1$ is an instance of the schema $x = y \And \phi^v_x \Yields \phi^v_y$. Thus, $\chi_0 \Yields \chi_1$ is equal to $\underline a = \underline b \And \phi^v_{\underline a} \Yields \phi^v_{\underline b}$ for some  $\mathbf \Sigma(L_{\omega \omega})$ formula $\phi$ with a single free variable $v$, and objects $a$ and $b$. Let $V$ be a verification of $\underline a = \underline b \And \phi^v_{\underline a}$. Then $V$ contains $\underline a = \underline b$, so $a$ and $b$ are the same object. It follows that the symbols $\underline a$ and $\underline b$ are equal, so the formulas $\phi^v_{\underline a}$ and $\phi^v_{\underline b}$ are also equal. Thus $V$ is a verification of both $\phi^v_{\underline a}$ and $\phi^v_{\underline b}$.
\end{proof}

\begin{lemma}[base theory]
Let $\chi_0$ and $\chi_1$ be $\mathbf \Sigma(L_{\omega \omega})$ formulas. Assume that $\chi_0 \Yields \chi_1$ is an application of a logical axiom, and that $\chi_0$ is true. Conclude $\chi_1$ is true.
\end{lemma}

\begin{proof}
Let $\phi \Yields \psi$ be the axiom that we are applying. Examining the proof of the preceeding lemma, we find a primitive recursive function that takes verifications of instances of $\phi$ to instances of $\psi$. Let $V_0$ be a verification of $\chi_0$. Let $V_1$ be the result of including, for each instance of $\phi$ in $V_0$, the verification for the corresponding instance of $\psi$, and adjoining, for each formula having an instance of $\phi$ as a subformula, the result of applying $\phi \Yields \psi$. Examining each element of $V_1$ we find that it is a verification.
\end{proof}

Now we show the validity of positivistic logic.

\begin{proposition}[base theory]
Let $\chi$ and $\chi'$ be $\mathbf \Sigma(L_{\omega \omega})$ formulas. Assume that $\chi \Yields \chi'$ has a positivistic proof using only logical axioms, and that $\chi$ is true. It follows that $\chi'$ is true.
\end{proposition}

\begin{proof}
Let $V$ be a verification of $\chi$. The procedure taking $V_0$ to $V_1$ in the proof of the preceding lemma is primitive recursive, so there is a function that assigns a verification to each sentence of the proof. We establish that each value of this function is a verification using $\Delta$ induction.
\end{proof}

Each step of a proof using only the logical axioms is uniquely an application of a logical axiom, except for trivial applications of logical axiom (21). However, where that is not the case, we may ask that the axiom being applied and the site of its application be specified at each step, as part of the definition of the proof.

\begin{lemma}[base theory]
Let $\phi(v_1, \ldots v_n)$ be a $\mathbf \Sigma(L_{\omega\omega})$ formula. The following is a theorem of the base theory: ``Let $a_1, \ldots, a_n$ be objects. Assume $\phi(\underline{a_1}, \ldots, \underline{a_n})$ is true. Conclude $\phi(a_1, \ldots, a_n)$.'' The converse is also a theorem.
\end{lemma}

Our base theory does not have any axioms that include parameters, so the parameters that appear in our formulas are essentially free variables. Thus, the assumption that $\phi(\underline{a_1}, \ldots, \underline{a_n}$ is true should be expressed with an explicit description of $\phi$, rather than using a parameter symbol. Later, we will supplement the base theory with all true $\mathbf \Sigma$ sentences, on the intuition that since truths may be verified, truths may be assumed, and our present caveat will be moot. We avoid supplementing the base theory now, to keep it countable. (Of course, we may add parameters just for the finite sets, which would make for a more aestheatic theory, but which would also slightly complicat our arguments.)

\begin{proof}[Proof of lemma.]
We apply induction on syntax to the the semantic properties of the Truth predicate. The construction of the proofs of the two theorems are primitive recursive.
\end{proof}

\begin{theorem}[base theory]
Let $\tau$ be a theory extending the base theory, and let $\phi \Yields \psi$ be one of its theorems. The following is a theorem of $\tau$: ``Let $\chi_0$ and $\chi_1$ be $\mathbf \Sigma(L_{\omega \omega})$ sentences. Assume that $\chi_0$ is true, and that $\chi_0 \Yields \chi_1$ is an application of $\phi \Yields \psi$. Conclude that $\chi_1$ is true.''
\end{theorem}

\begin{proof}
We argue in $\tau$: ``Let $\chi_0$ and $\chi_1$ be $\mathbf \Sigma(L_{\omega \omega})$ sentences. Assume that $\chi_0$ is true, and that $\chi_0 \Yields \chi_1$ is an instance of $\phi \Yields \psi$. By definition of instance, there are objects $a_1, \ldots, a_n$ such that $\chi_0$ is $\phi(\underline a_1, \ldots , \underline a_n)$, and $\chi_1$ is $\psi(\underline a_1, \ldots , \underline a_n)$. Since $\phi(\underline a_1, \ldots , \underline a_n)$ is true, $\phi(a_1, \ldots, a_n)$, so $\psi(a_1, \ldots, a_n)$. We conclude that $\psi(a_1, \ldots, a_n)$ is true.'' This argument includes the given proof of $\phi \Yields \psi$, and it includes the truth schema established in the preceding lemma.

We argue again in $\tau$: ``Let $\chi_0$ and $\chi_1$ be $\mathbf \Sigma(L_{\omega \omega})$ sentences. Assume that $\chi_0$ is a true, and that $\chi_0 \Yields \chi_1$ is an application of $\phi \Yields \psi$. Let $V_0$ be a witness of $\chi_0$. Let $V_1$ be the result of including, for each instance of $\phi$ in $V_0$, a verification of the corresponding instance of $\psi$, and adjoining, for each formula having an instance of $\phi$ as a subformula, the result of applying $\phi \Yields \psi$. We may check that $V_1$ is a verification by examining its elements. We conclude that $\chi_1$ is true.'' We insert our previous argument in $\tau$ when we we include, for each instance of $\phi$ in $V_0$, a verification of the corresponding instance of $\psi$. That argument does not give us a primitive recursive procedure for obtaining verifications of instances of $\psi$. This argument obtains the verification $V_1$ by applying $\Sigma$ collection.
\end{proof}

The second part of the proof demonstrates that the correctness of a theory $\tau_1$ is equivalent to ``Let $\chi_0 \Yields \chi_1$ be a substitution instance of an axiom of $\tau_1$. Assume $\chi_0$ is true. Conclude $\chi_1$ is true.''

\begin{corollary}[base theory]\label{block: correctness}
Let $\tau$ be a $\mathbf \Sigma(L_{\omega \omega})$ theory extending the base theory, and let $\tilde \tau$ be $\tau$ together with the conditional expressing the correctness of $\tau$. The theory $\tilde \tau$ proves the correctness of $\tilde \tau$.
\end{corollary}

\begin{proof}
The theory $\tilde \tau$ proves the correctness of $ \tau$ by definition. The theory $\tilde \tau$ proves the correctness of the new axiom by the above theorem, because the new axiom is trivially one of its theorems.
\end{proof}

The above corollary is at the same time reassuring and alarming. It is reassuring that we can prove the correctness of $\tilde \tau$ without additional assumptions, but what of G\"odel's second incompleteness theorem? G\"odel's second incompleteness theorem is most certainly provable in our base theory. The kernel of its proof is so elementary, that it applies to any system that allows even the most basic finitary procedures. 

The usual proofs of G\"odel's first and second incompleteness theorems need to be modified only slightly to fit the positivistic framework; the G\"odel sentence, which is nonpositivistic, must be replaced by its logical complement. The heart of the matter is that diagonalization is only possible for positivistic predicates. Diagonalization yields a sentence $\sigma$ that is logically equivalent to $\PP(\sigma \Yields \Falsehood)$.

There is a proof from $\sigma$ to $\PP(\Truth \Yields \Falsehood)$: ``Assume $\sigma$. It follows that $\PP(\Truth \Yields \sigma)$, because $\sigma$ is a positivistic arithmetical sentence, and that $\PP(\sigma \Yields \Falsehood)$, by construction of $\sigma$. Conclude that $\PP(\Truth \Yields \Falsehood)$.'' Assume consistency, i. e., $\PP(\PP(\Truth \Yields \Falsehood) \Yields \Falsehood)$. We have just shown $\PP(\sigma \Yields \PP (\Truth \Yields \Falsehood))$. It follows that $\PP(\sigma \Yields \Falsehood)$, i. e, that $\sigma$. Since $\sigma$ is a positivistic arithmetical sentence, we infer $\PP(\Truth \Yields \sigma)$, so we may conclude $\PP(\Truth \Yields \Falsehood)$ as desired.

The resolution of this apparent contradiction is that we have shown the correctness of $\tilde \tau$, and not its validity. Thus, we do not have a proof of consistency. We cannot apply our argument for the validity of the theory of positivistic logic, because there is no primitive recursive function producing witnesses. We cannot appeal to $\Sigma$ induction because it isn't among our axioms; it \emph{cannot} be among our axioms because of its high logical complexity. It is generally not admissible.

Thus, $\tilde \tau$ cannot prove that every proof is valid; it can only prove that every step is correct. This, however essentially guarantees that every proof that we write down is valid: for every natural $n$, we have a simple proof of the validity of proofs of fewer than $n$ steps, linearly in $n$. This is a very natural variant of Pudl\'ak's result \cite{Pudlak} that proofs of finitary consistency exist with size linear in the size of the proofs considered.

In the finitary setting we may take $\tau$ to consist of the \emph{theorems} of $\PRA$. Thus, $\tilde \tau$ is obtained by adding the \emph{validity} of $\PRA$. The theory $\tilde \tau$ is equivalent to $\PRA$ together with the validity of $\PRA$. The theory $\PRA$ proves Parson's theorem, and it proves the reduction of $\WKL$ to $\mathsf{ I \Sigma_1}$; so, the theory $\tilde \tau$ proves its own correctness and the validity of $\WKL$. When I adopt the mindset of a finitist, I find the theory $\tilde \tau$ just described to be quite acceptable, though it does go a hair further than $\PRA$ itself, which, after Tait's analysis, is often identified with finitism.

\section{the universe of pure sets}\label{section: universe}

We have so far considered model universes within a cumulative hierarchy described by the axioms of $\ZFC$. We now consider the genuine universe of pure sets: the totality of all potential pure sets. This universe of all pure sets may be like one of the models we have considered; it may consist of hereditarily finite sets, or of hereditarily countable sets, or it may permit the powerset construction. Since the axioms of $\ZFC$ are not all $\Pi_2$, we might suppose that our investigation of these model universes occurred within some transitive model of $\ZFC$.

Our speculation about the nature of the full universe is guided by two principles. First, the truth of any predicate should be determined by examining membership following some logical scheme. Thus, the totality of procedures of the full universe should be essentially exhausted by $\mathbf \Sigma$ formulas, with no additional nonlogical symbols. Each $\mathbf \Sigma(L_{\infty \omega})$ formula is equivalent to a $\mathbf \Sigma(L_{\omega \omega})$ formula, so our notion of procedure does not depend on this choice of logic. 

Our second principle is that the universe should be complete with respect to the sets that may exist. The simplest interpretation of this principle is the completeness of some $L_{\infty\omega}$ system, e. g., $\lL\kK_{\infty \omega}$. A special case of this principle is more deeply compelling: we ask that every consistent description of a unary predicate on a transitive set is satisfied by some subset of that transitive set. This is analogous to constructibility, except instead of exhibiting the predicate that defines the new set, we describe it indirectly.

\begin{definition}[base theory] The \underline{set completeness principle}:
Let $A$ be a transitive set. Let $S$ be a unary predicate symbol. Let $T$ be a $\mathbf K (L_{\infty \omega}(= , \in, S))$ theory with parameters from $A$, that includes the following axioms.
\begin{enumerate}
\item The equality axioms:
\begin{enumerate}
\item $\forall x\: x = x$
\item $\forall x\: \forall y\: x = y \Iff y = x$
\item $\forall x\: \forall y\: \forall z\: x = y \And y =z \Implies x =z$
\item $\forall x\: \forall y\: x = y \And S(x) \Implies S(y)$
\end{enumerate}
\item The atomic axioms:
\begin{enumerate}
\item $\underline a \in \underline b$, for $a \in b$
\item $ \neg \underline a \in \underline b$, for $a \not \in b$
\item $\underline a  = \underline b$, for $a = b$
\item $\neg \underline a = \underline b$, for $a \neq b$
\end{enumerate}
\item The content axiom: $\forall x \: \OR_{a \in A} x = \underline a$
\end{enumerate}
Conclude that either $T$ is inconsistent for $\lL\kK_{\infty \omega}$ (figure 4), or there is a subset $B \subsetof A$ such that $(A, =, \in, B)$ is a model of $T$.
\end{definition}

For an arbitrary structure: the equality axioms are the substitution axioms for atomic formulas, together with the equivalence relation axioms for equality; the atomic axioms are the true atomic sentences, and the true negated atomic sentences; and, the content axiom is the universally closed disjunction expressing that every element is named by some closed term. The inconsistency of $T$ for $\lL\kK_{\infty \omega}$ means the derivability of $\AND T \Proves \emptyset$ in $\lL\kK_{\infty \omega}$.

\begin{definition}[base theory] The \underline{model completeness principle}:
Let $T$ be a $\mathrm K(L_{\infty \omega}(\S))$ a theory whose vocabulary $\S$ consists of finitary function and predicate symbols. Conclude that either $T$ is inconsistent for $\lL\kK_{\infty \omega}$, or $T$ has a model.
\end{definition}

The model completeness principle is not as well motivated by the notion of a full universe of sets, but it may be motivated by the notion of a full universe of structures. The very presence of infinite sets in mathematical discourse is arguably to ensure that any consistent finitary theory has model. Notably, the model completeness principle determines the cardinality structure of the universe: infinite sets exist, and every set is countable.

Even the weaker set completeness principle implies that every set is countable. Unlike the model completeness principle, the set completeness principle holds in the model universe $H_\omega$ of hereditarily finite sets. In the presence of infinite sets, the set completeness principle produces generic sets. Let $P$ be a forcing partial order, let $A$ be its transitive closure, and let $C$ be a nonempty collection of dense subsets of $P$. Forcing semantics shows that the following axioms are consistent with the previously given axioms:
\begin{enumerate}\setcounter{enumi}{3}
\item $\neg S(\underline a)$, for $a\not \in P$
\item $S(\underline p) \Implies S(\underline q)$, for $p, q \in P$ such that $p \leq q$
\item $S(\underline p) \And S(\underline q) \Implies \OR_{r \leq p, q} S(\underline r)$, for $p, q \in P$
\item $ \OR_{p \in D} S(\underline p)$, for $D$ in $C$
\end{enumerate}
In the usual way, we obtain a surjection from $\omega$ onto any given set.

\begin{theorem}[base theory]\label{block: equivalence}
The model completeness principle is equivalent to the set completeness principle with the axiom of infinity.
\end{theorem}

The proof is unexpectedly long for such an expected result. The main difficulty is that the derivability of a sequent is \textit{a priori} not decidable. Thus, we cannot form sets of derivable sequents. We cannot appeal to Barwise completeness, because the universe is not assumed to be a transitive set in a model of $\ZFC$.

\begin{proof}
The forward direction is simpler. The model completeness principle implies the set completeness principle because the equality axioms, the content axiom, and the atomic axioms ensure that any model is canonically isomorphic to an expansion of $(A, = , \in)$. The model completeness principle implies the axiom of infinity because our base theory is more than sufficient to show that Robinson arithmetic $\mathsf{Q}$ is consistent, and this theory cannot have a finite model as the successor function is an injection that is not a surjection.

The backward direction is more involved. We fix a set $T$ of $\mathrm K(L_{\infty \omega}(\S))$ sentences, for some set $\S$ of finitary function and predicate symbols; we will show that $T$ is inconsistent or it has a model. Without loss of generality we may assume that $\S$ consists entirely of predicate symbols, by replacing function symbols by predicate symbols, as usual. Thus, we assume that $\S = \{P_0, P_1, \ldots \}$ is a countable set of predicate symbols. Let $A$ be the set of hereditarily finite sets. We can interpret $T$ as a set of $\mathbf K (L_{\infty \omega}(= , \in, S))$ sentences, with parameters from $A$, by interpreting each atomic formula $P_n(x_1, \ldots, x_m)$ as the formula ``there is an object that is an $(m+1)$-tuple $(n, x_1, \ldots, x_m)$ that satisfies $S$''. 

We define $T_1$ to be the $\mathbf K(L_{\infty \omega}(=, \in, S))$ theory with parameters from $A$, obtained by adding the equality axioms, the content axiom, and the atomic axioms to these interpreted axioms of $T$. If the vocabulary $\S$ includes a binary equality relation symbol, we treat this relation symbol as distinct from the equality symbol in $T_1$; thus it is an equivalence relation that may fail to satisfy substitution for some $\mathbf K (L_{\omega \omega}(= , \in, S))$ formulas. We code tuples as partial functions whose domain is a finite ordinal; we exclude $0$-tuples. Note that every element of $A$ that is a tuple is uniquely such, that no tuple is a natural number, and that $A$ has infinitely many elements that are neither tuples nor natural numbers. By the set completeness principle, the theory $T_1$ either is inconsistent or is modelled by some subset of $A$. By construction of $T_1$, any such model may be interpreted as a model of $T$; since any equality symbol in the vocabulary of $T$ is treated as a nonlogical symbol, we expect $T$ to have an infinite model; since the set completeness principle implies that every set is countable, we can expect $T$ to have a countable model. It remains to argue that if $T_1$ is inconsistent, then $T$ is inconsistent. We assume that $T_1$ is inconsistent, and we proceed by repeatedly modifying $T_1$ in minor way, and checking that the result is still be inconsistent.

We define $T_2$ by enlarging the vocabulary of $T_1$ to include an $n$-ary $n$-tuple formation function symbol for each positive integer $n$, and then interpreting the axioms of $T$ as we have done in the definition of $T_1$, but using these new function symbols to translate the formulas of $T$; we include atomic axioms and substitution axioms with these new function symbols. Because we are working with the infinitary deductive system $\lL\kK_{\infty \omega}$, the atomic axioms, together with the content axiom, are enough to prove every true $\mathbf K(L_{\infty \omega}( = , \in, S, \<\cdot\>, \< \cdot, \cdot\> ,\< \cdot, \cdot, \cdot\>, \ldots))$ sentence, and in particular that the new function symbols do form tuples. Thus $T_2$ derives every axiom of $T_1$, so $T_2$ is inconsistent.

We define $T_3$ by excluding all parameters that name tuples in $A$, and removing the axioms where they occur from $T_2$, rewriting the content axiom to express that every object is equal to some closed term. This step may be achieved simply by replacing each parameter that names a tuple in $A$ by the closed term expressing that tuple in terms of parameters that do not name tuples in $A$. This second presentation shows that we may obtain a deduction of the inconsistency of $T_3$, by taking a deduction of the inconsistency of $T_2$, and performing such a replacement throughout. We note that every element of $A$ is the value of a unique term whose parameter symbols do not name tuples in $A$.

We define $T_4$ by excluding the membership relation symbol $\in$, and removing the atomic and substitution axioms for $\in$ from $T_3$. We obtain a deduction witnessing the inconsistency of $T_4$ by applying the interpolation theorem, with sentences containing $\in$ on the left, and sentences containing $S$ on the right. The usual proof of Craig's interpolation theorem from the cut-elimination theorem applies. Recall that we rewrite all the sentences with only atomic formulas negated, and then we obtain a cut-free derivation where negation is applied immediately, by pushing applications of the negation rules in a cut-free derivation towards initial sequents. We then obtain interpolants recursively over the proof tree; negation is only applied in the familiar case where we have sequents of literals, and after negation is applied, interpolants are simply combined in the obvious way. The interpolant we obtain is a sentence that contains neither $\in$ nor $S$, i. e., whose only only predicate symbol is equality (as well as $\Truth$ and $\Falsehood$). The sentences appearing on the left are true sentences about the structure $(A, =, \in)$, so the interpolant is also true about this structure. Thus, the interpolant may be deduced using just the atomic axioms and substitution axioms for $=$, and therefore, from $T_4$. Bringing the sentences on the right of our sequent back to the left, we conclude that $T_4$ is inconsistent.

We will apply the cut-elimination theorem to the deduction of a contradiction in $T_4$. There is a standard variant of cut-elimination that makes accommodations for equality; it implements the equality axioms as initial sequents. This is the approach taken by Feferman in his syntactic proof of infinitary cut-elimination \cite{Feferman}. The deductive system without equality is thus extended by two rules: the first deduces a sequent of the form $\emptyset \Proves t = t$ from nothing; the second deduces sequents such as $s_1 = s_2, t_1 = t_2, t_2 = t_3,  P(s_1, t_1) \Proves P(s_2, t_3)$ from nothing, with the set of equalities in the antecedent called an equality chain. I believe that Feferman intended to include the the reverse inequalities $s_{i+1} = s_i$ in his definition of an equality chain, as there is otherwise no way to derive the sequent $x = y \Proves y =x$ without using cut. We will apply a stronger variant of cut-elimination that implements the true inequalities between closed terms as initial sequents. Specifically, for any pair of distinct closed terms $s()$ and $s'()$, we allow
$$ s()= s'() \Proves \emptyset$$
as an initial sequent. To eliminate the cut rule, we also allow initial sequents whose antecedents consist of equality chains connecting distinct closed term $s()$ and $s'()$, and whose consequent are empty. These new additions require only a minor addition to Feferman's proof, of the same kind as is made for the other initial sequents.

Thus, we have a deduction in this extended system of the sequent
$$\forall x\: \OR_{s()} x = s(), \AND \mathrm{Atomic}, \AND \mathrm{Translations} \Proves \emptyset$$
where the disjunction is taken over all closed terms $s()$, and $\mathrm{Atomic}$ is the set of all atomic axioms, and $\mathrm{Translations}$ is the set of translations of the axioms of $T$. Since the parameters naming atomic tuples are not in $T_4$, the atomic axioms are equalities between equal closed terms, and inequalities between distinct closed terms. Both kinds of sentences follow from our initial sequents, so by the cut rule, we have a deduction of 
$$\forall x\: \OR_{s()} x = s(), \AND \mathrm{Translations} \Proves \emptyset.$$
Applying the cut-elimination theorem to this derivation, we obtain a deduction that is cut-free. We may assume that any disjunction $\OR x = s()$ becomes universally closed immediately after it forms on the left.

We now show by induction over the fixed cut-free derivation that each sequent in the derivation has a derivation even when the content axiom $\forall x\: \OR x = s()$ is excluded from its antecedent. The only nontrivial steps are where the content axiom is formed by the left disjunction and left universal quantification rules; there is nothing to do if it is formed by the left weakening rule. Assume that the content axiom is just formed by successive applications of the left disjunction rule and the left universal quantification rule, producing some sequent $\forall x \: \OR x = s(), \Gamma \Proves \Delta$ in the fixed cut-free derivation. Write $\Gamma_0$ for $\Gamma$ with all instances of the content axiom removed. By the induction hypothesis, the sequent
$$t(x_1, \ldots, x_n) = s(), \Gamma_0 \Proves \Delta$$
has a derivation for some term $t(x_1, \ldots, x_n)$ and all closed terms $s()$. The formulas in $\Gamma_0$ and $\Delta$ are equalities and substitution instances of subformulas of translations of axioms of $T$. Since $\Gamma_0$ and $\Delta$ are finite, and infinitely many parameters are excluded from translations of axioms of $T$, there is a parameter $a$ that does not appear in any formula in $\Gamma_0$ or $\Delta$. Thus, we have derivations of the following sequents:
$$t(x_1, \ldots, x_n) = a, \Gamma_0 \Proves \Delta$$
$$\exists y\: t(x_1, \ldots, x_n) =y, \Gamma_0 \Proves \Delta$$
$$\Gamma_0 \Proves \Delta$$
By induction, we obtain a derivation of $\AND \mathrm{Translations} \Proves \emptyset$, and therefore a cut-free derivation of this sequent. Since equality does not occur in any translation of an axiom of $T$, the subformula property implies that this sequent may be derived in the equality-free system. 

We define $T_5$ to be the theory whose axioms are translations of the axioms of $T$; we have shown that its inconsistency sequent $\AND \mathrm{Translations} \Proves \emptyset$ has a cut-free derivation. Every atomic subformula in the translation of an axiom of $T$ is of the form $S((\underline n, x_1, \ldots, x_m))$ for some integers $n$ and $m$, so every atomic subformula of any formula in any sequent in the cut-free derivation is a substitution instance of $S((\underline n, x_1, \ldots, x_m))$. We now follow the intuition that we may treat any function symbols appearing below a top-most function symbol as  distinct from each top-most function symbol, and that these distinct new function symbols may name constant functions; so, we may replace any nonvariable term below a top-most function symbol by the same new constant symbol. Formally, we simply replace every nonvariable term appearing as the argument of a top-most function symbol after the first position with the same new variable, and check that the cut-free deduction is still a cut-free deduction.

We now have a cut-free deduction of inconsistency in $T_5$ with the property that every atomic subformula is of the form $S((\underline n, v_1, \ldots v_m))$ for some natural number $n$ and variables $v_1, \ldots, v_m$. We may simply replace each such formula by $P_n(v_1, \ldots, v_n)$ to obtain a proof of incoconsistency in $T$.
\end{proof}

The principle that every set is countable is not as restrictive as it appears. Under $\ZFC$, every recursively enumerable $\mathrm K(L_{\omega \omega}(= , \in))$ theory that has a transitive model, has a countable transitive model. Thus, under various standard large cardinal assumptions, there are countable transitive models of $\ZFC$ satisfying various standard large cardinal assumptions. So, our base theory with the model completeness principle is consistent with the existence of such transitive models.

The same sort of argument may be made for the axiom of constructibility. If a recursively enumerable $\mathrm K(L_{\omega \omega}(= , \in))$ theory has a transitive model, it has a countable transitive model, so by the Shoenfield absoluteness theorem, it has a countable transitive model in $L$. The set of hereditarily countable sets in $L$ models our base theory, and the model completeness principle, and the axiom of constructibility. I do not view the axiom of constructibility to be as natural an axiom for the full universe of pure sets as the completeness principles discussed, but it does preserve some features of the standard picture of the universe. The axiom of constructibility stratifies the universe of sets into a cumulative hierarchy, and I suppose that it expresses a form of transfinite predicativism. In this setting, it is natural to consider ``large ordinal'' axioms of the form ``$L_{\alpha}$ is a model of $\ZFC + \cdots$''.

The reasoning here echos Hamkins's arguments for principles expressing an extensible concept of set, including that every universe should be a transitive model in a universe satisfying the axiom of constructibility. My impression is that Hamkins advocates for a multiverse view of set theory, whereas these notes advocate for a universe view. However, a multiverse view in which well-foundedness is absolute may be reconciled with the universe view expressed here by saying that the former focuses on the many transitive models of $\mathsf{ZF}$, whereas the latter focuses on the set-theoretic universe as a whole.

\section{assertibility}\label{section: assertibility}

The $L_{\omega \omega}$ completeness theorem implies that any theorem, that can be derived from the axioms of a positivistic theory $T$ via classical logic, also has a positivistic proof from that theory. Furthermore, the syntactic proof of the cut-elimination theorem yields a primitive recursive procedure that transforms any given classical derivation of a conditional into a positivistic proof. However, as is well known, the elimination of cuts can produce a superexponential growth in the length of a proof.

The cut-elimination theorem for $\lL\kK_{\omega \omega}$ can be proved in $\PRA$, so it is certainly a theorem of our base theory. Thus, we may first prove the existence of a derivation in classical logic by exhibiting such a derivation, and then infer the existence of a positivistic proof. However, the existence of a proof does not establish the proved conditional without an additionally accepted validity principle. For a theory that proves its own correctness, the acceptance of such a validity principle amounts to an instance of the $\Sigma_1$-induction rule on $\omega$. Thus, a minute extension of the base theory shows that the theorems proved from the axioms of the base theory using classical logic are valid. The same observation holds for the nominally weaker intuitionistic logic, of course. However, intuitionistic derivability is the more interesting notion because it is constructive.

We introduce a new notion of assertibility as a form of semantics for intuitionistic formulas. It is inspired by the notion of assertibility in Weaver's \textit{Truth and Assertibility}, but there are two significant differences. First, this notion of assertibility is defined relative to a theory $\tau$. Second, this notion of assertibility does not satisfy the full capture schema $\phi \Yields \AA(\phi)$, though it does satisfy the other axioms. Both of these differences may be rationalized conceptually. The positivistic approach recognizes only $\mathbf \Sigma$ formulas as meaningful; it does not recognize any conditionals as absolutely justifiable, only justifiable relative to given axioms. I have no intuition that there is any complete objective notion of ``conclusive demonstration'' \cite[p. 89]{Weaver}. Weaver justifies the capture law, that $A$ implies the assertibility of $A$, on the principle that ``all truths can be known'' \cite[p. 109]{Weaver}. We do not depart from this principle; we simply have a stricter conception of what sentences may be truths. Indeed, for any $\mathbf \Sigma$ sentence $\phi$, we will have that $\phi$ implies the assertibility of $\phi$.

\begin{definition}[base theory + every set is countable]\label{block: assertibility definition}
Let $\tau$ be an intentional $ \mathbf \Sigma (L_{\omega \omega})$ theory extending the base theory. For a sentence $\phi \in \mathbf I(L_{\infty \omega})$ we say that $\phi$ is \underline{assertible} from $\tau$, and write $\AA_\tau(\phi)$, in case there is a set $A$ of axioms of $\tau$ (``axioms''), and a set of true $\mathbf \Sigma(L_{\infty \omega})$ sentences $E$ (``observations'') such that $ \AND \overline A, \AND E \proves \phi$ is derivable in the infinitary intuitionistic sequent calculus $\lL\iI_{\infty \omega}$ (figure 5).
\end{definition}

\begin{proposition}[base theory + every set is countable]
Let $\phi \Yields \psi$ be an axiom of $\tau$. Conclude that $\overline{\phi \Yields \psi}$ is assertible.
\end{proposition}

\begin{proof}
Immediate from the definition.
\end{proof}

\begin{proposition}[base theory + every set is countable]
Let $\phi_1, \ldots, \phi_n$ be assertible $\mathbf I(L_{\infty \omega})$ sentences. Assume that the sequent $\phi_1, \ldots, \phi_n \Proves \psi$ is derivable in $\lL\iI(L_{\infty \omega})$ for some $\mathbf I(L_{\infty \omega})$ sentence $\psi$. Conclude that $\psi$ is also assertible.
\end{proposition}

\begin{proof}
Combine the axioms, and observations for $\phi_1, \ldots, \phi_n$.
\end{proof}

We will show that assertible $\mathbf \Sigma(L_{\omega \omega})$ sentences are true, assuming the iterated validity of the theory $\tau$. Essentially, we will appeal to a validity principle each time we combine proofs of sentences $\phi(\underline a)$ for $a \in b$ into a proof of $\forall x \in \underline b\: \phi(x)$. For each theory $\tau$, it is natural to consider the stronger theory $\tau'$ obtained by adding the validity of $\tau$ to $\tau$. This process may be naturally iterated through the ordinals, by defining $\tau^{(\alpha)}$ to be $(\tau^{(\alpha-1)})'$, when $\alpha$ is a successor ordinal, and $\tau^{(\alpha)} = \Union_{\beta \in \alpha} \tau^{(\beta)}$, when $\alpha$ is a limit ordinal. Since we are working with intentional theories, this description is informal. Furthermore, our formalization diverges for this description because the absence of induction complicates the normally trivial observation that higher-indexed theories are stronger.

\begin{definition}[base theory]
Let $\tau$ be a $\mathbf \Sigma(L_{\omega \omega})$ theory. For each ordinal $\beta$, we define $\tau^{(\beta)}$ to be $\tau$ together with the validity of $\tau^{(\alpha)}$, for each $\alpha \in \beta$.
\end{definition}

The most concise way to formalize this definition is as a fixed point formula $\tau(\alpha, \phi \Yields \psi)$; we write $\tau^{(\alpha)}(\phi \Yields \psi)$ for $\tau(\underline \alpha, \phi \Yields \psi)$. For each ordinal $\alpha$, the formula $\tau^{(\alpha)}(\phi \Yields \psi)$ is generally strictly $\mathbf \Sigma(L_{\omega \omega})$, but of course the function taking each ordinal $\alpha$ to $\tau^{( \alpha)}$ is primitive recursive. It is immediate that any theorem of $\tau^{(\alpha)}$ is a theorem of $\tau^{(\beta)}$, whenever $\alpha \in \beta$. Naturally, we define $\tau^{(\Omega)}$ to be $\tau$ together with the validity of $\tau^{(\alpha)}$, for each ordinal $\alpha$.

\begin{definition}[base theory]
Let the \underline{theory of observations} $\sigma$ be the theory of true $\mathbf \Sigma(L_{\omega \omega})$ sentences, i. e., the theory whose axioms are conditionals $\Truth \Yields \varepsilon$ for true $\mathbf \Sigma(L_{\omega \omega})$ sentences $\varepsilon$.
\end{definition}

It is natural to consider a $\Sigma(L_{\omega \omega})$ theory $\tau$ in combination with the theory $\sigma$. If the provability of a conditional in $\tau$ justifies that conditional on the basis of reason, then the provability of a conditional in $\tau + \sigma$ justifies that conditional on the basis of reason \emph{and fact.} A finitary agent in an infinitary universe cannot verify sentences directly, but if they recognize the validity of a parameter-free theory $\tau$, they may quickly prove the validity of $\tau + \sigma$. In a sense, the finitary agent proves that an infinitary agent may take the true $\mathbf \Sigma(L_{\omega \omega})$ to be axioms.

\begin{lemma}[base theory + every set is countable]
Let $\tau$ be a $\mathbf \Sigma(L_{\omega \omega})$ theory extending the base theory. Let $A$ be a set of axioms of $\tau$; let $E$ be a set of true $\mathbf \Sigma(L_{\infty \omega})$ sentences; and let $\phi_1\ldots \phi_n , \psi$ be $\mathbf \Sigma(L_{\omega \omega})$ formulas. Assume that the sequent $\AND \overline A, \AND E, \phi_1, \ldots, \phi_n \Proves \psi$ is derivable in $\lL\iI_{\infty \omega}$. Conclude that the conditional $\phi_1 \And \ldots \And \phi_n \Yields \psi$ is provable in $(\tau+ \sigma)^{(\Omega)}$.
\end{lemma}

\begin{proof}
We first observe that the true $\mathbf \Sigma(L_{\infty \omega})$ formulas are derivable from a simple subclass of such formulas: the atomic axioms and the content axioms. The \underline{atomic axioms} are simply the true atomic sentences; recall that we are treating $\neq$ as relation symbol rather than as an abbreviation. The \underline{content axioms} are the universally closed disjunctions: $\forall x\in \underline b \: \OR_{a \in b} x = \underline a$. To simplify the presentation, we render each content axiom as the derivably equivalent formula $\forall x\: \OR  \{x \not \in \underline b\} \union \{x = \underline a \suchthat a \in b\}$.

We assume the existence of a derivation $\AND \overline A, \AND E, \phi_1, \ldots, \phi_n \Proves \psi$. It follows that there is a derivation of $\AND \overline A, \AND E_{\text{atomic}}, \AND E_{\text{content}}, \phi_1, \ldots, \phi_n \Proves \psi$, for some set $E_{\text{atomic}}$ of atomic axioms, and some set $ E_{\text{content}}$ of content axioms. By the cut-elimination theorem, there is a cut-free derivation of this sequent. We inductively construct, for each sequent in the derivation, a proof from the conjunction of the $\mathbf \Sigma(L_{\omega \omega})$ formulas in the sequent's antecedent to the $\mathbf\Sigma (L_{\omega \omega})$ formula in its succedent. By induction, any formula in the succedent of a sequent in the fixed derivation must be $\mathbf \Sigma (L_{\omega \omega})$. If the succedent is empty, we instead construct a proof to $\Falsehood$. Note that our use of induction is justified because the argument is constructive.

The following cases are trivial: the axiom rule, the structural rule, the weakening rules, the truth rule $(\proves \Truth)$ and the falsehood rule $(\Falsehood \Proves)$. The right implication rule $(\Proves \Yields)$ does not occur. In the left implication rule $(\Yields \Proves)$ case, the new conditional must be an instance of an axiom; concatenating the proofs for the two input sequents with this axiom, we obtain a proof for the output sequent.

In cases where the rule application produces a $\mathbf \Sigma(L_{\omega \omega})$ formula, our argument appeals to deep inference, following the pattern in our proof of completeness. This applies to the conjunction rules $(\AND \Proves)(\Proves \AND)$, the disjunction rules $(\OR \Proves)(\Proves \OR)$, the existential quantification rules $(\exists \Proves)(\Proves \exists)$, and the universal quantification rules $(\forall \Proves)(\Proves \forall)$. We therefore consider the cases in which an application of one of these rules does not produce a $\mathbf \Sigma (L_{\omega \omega})$ formula; such an application cannot occur on the right.

If the left conjunction rule $(\AND \Proves)$ produces a non-$\mathbf \Sigma(L_{\omega \omega})$ formula, then that formula is a conjunction of universally closed axioms of $\tau$, of content axioms, or of atomic axioms. In the first two cases, the $\mathbf \Sigma(L_{\omega \omega})$ formulas in the output sequent are exactly the same as the $\mathbf \Sigma(L_{\omega \omega})$ of the input sequent, so the proof for the input sequent works for the output sequent. In the third case, the loss of an atomic assumption $\epsilon$ is innocuous, because $\Truth \Yields \epsilon$ is a theorem of $\sigma$.

If the left universal quantification rule $(\forall \Proves)$ produces a non-$\mathbf \Sigma (L_{\omega \omega})$ formula, then the universal quantifier is applied to a non-$\mathbf \Sigma(L_{\omega \omega})$ input formula. Indeed, either the input formula includes an implication symbol $\Yields$, or the input formula is of the form $\OR\{t \not \in \underline b\} \union \{t = \underline a \suchthat a \in b\}$. In either case, the proof for the input sequent works for the output sequent. The left existential quantification rule $(\Exists \Proves)$ only produces $\mathbf \Sigma (L_{\omega \omega})$ formulas.

If the left disjunction rule $(\OR \Proves)$ produces a non-$\mathbf \Sigma(L_{\omega \omega})$ formula, then that formula is of the form $\OR \{t \not \in \underline b\} \union \{t = \underline a \suchthat a \in b\}$. Let $x_1, \ldots, x_m$ be the free variables in the output sequent; let $\phi_1(x_1, \ldots, x_m), \ldots, \phi_n(x_1, \ldots, x_m)$ be the $\mathbf \Sigma(L_{\omega \omega})$ formulas on the left of the output sequent; and let $\psi(x_1, \ldots, x_m)$ be the $\mathbf \Sigma(L_{\omega \omega})$ formula on the right of the output sequent, or the formula $\Falsehood$ if the there is no formula on the right of the output sequent. Thus, there is a least ordinal $\alpha$ such that 
$$ t(x_1, \ldots, x_m)  \not \in \underline b \And \phi_1(x_1, \ldots, x_m) \And \ldots \And \phi_n(x_1, \ldots, x_m) \Yields \psi(x_1, \ldots, x_m)$$
has a proof from $(\tau + \sigma)^{(\alpha)}$, and for all $a \in b$,
$$ t(x_1, \ldots, x_m) = \underline a \And  \phi_1(x_1, \ldots, x_m) \And \ldots \And \phi_n(x_1, \ldots, x_m) \Yields \psi(x_1, \ldots, x_m)$$
has a proof from $(\tau + \sigma)^{(\alpha)}$.
It follows that 
$$
\forall a \in \underline b \:\PP^{( \alpha)}(t(x_1, \ldots, x_m) = \underline a \And \phi_1(x_1, \ldots, x_m) \And  \ldots \And \phi_n(x_1, \ldots, x_m) \Yields \psi(x_1, \ldots, x_m))
$$
is a true $\mathbf \Sigma (L_{\omega \omega})$ formula, where $\PP^{(\alpha)}$ is the provability predicate for the theory $(\tau + \sigma)^{(\alpha)}$.
We now construct a proof of the desired conditional using the axioms of $(\tau +\sigma)^{(\alpha+1)}$:
\hspace{-100pt}\begin{align*}
&\phi_1(x_1, \ldots, x_m)  \And \ldots \And \phi_n(x_1, \ldots, x_m)
\\ & \Yields
[ t(x_1, \ldots, x_m) \not \in \underline b \And \phi_1(x_1, \ldots, x_m)\And \ldots \And \phi_n(x_1, \ldots, x_m)]
\\  &\qquad\qquad
\Or
[ t(x_1, \ldots, x_m) \in \underline b \And \phi_1(x_1, \ldots, x_m)\And \ldots \And \phi_n(x_1, \ldots, x_m)]
\\ & \Yields
\psi(x_1, \ldots, x_m)
\Or
[ t(x_1, \ldots, x_m) \in \underline b \And \TT(\phi_1(\underline x_1, \ldots, \underline x_m)) \And \ldots \And \TT(\phi_n(\underline x_1, \ldots, \underline x_m) )
\\  &\qquad\qquad\And\forall a \in \underline b\: \PP^{( \alpha)}(  t(x_1, \ldots, x_m)= \underline a \And \phi_1(x_1, \ldots, x_m) \And \ldots \And \phi_n(x_1, \ldots, x_m) \Yields \psi(x_1, \ldots, x_m)]
\\ & \Yields
\psi(x_1, \ldots, x_m)
\Or
[\TT(\phi_1(\underline x_1, \ldots, \underline x_m) \And \ldots \And \phi_n(\underline x_1, \ldots, \underline x_m) )
\\  &\qquad\qquad\And\PP^{( \alpha)}( t(x_1, \ldots, x_m) = \underline{ t(x_1, \ldots, x_m)} \And \phi_1(x_1, \ldots, x_m) \And \ldots\And  \phi_n(x_1 , \ldots, x_m) \Yields \psi(x_1, \ldots, x_m)]
\\ & \Yields
\psi(x_1, \ldots, x_m)
\Or
[\TT(\phi_1(\underline x_1, \ldots, \underline x_m) \And \ldots \And \phi_n(\underline x_1, \ldots, \underline x_m) ) \And \TT(t(\underline x_1, \ldots, \underline x_m) = \underline{t(x_1, \ldots, x_m)})
\\  &\qquad\qquad \And\PP^{( \alpha)}(t(\underline x_1, \ldots, \underline x_m) = \underline{ t(x_1, \ldots, x_m)} \And \phi_1(\underline x_1, \ldots,\underline x_m) \And \ldots\And \phi_n( \underline x_1, \ldots, \underline x_m) \Yields \psi(\underline x_1, \ldots,\underline x_m)]
\\ & \Yields
\psi(x_1, \ldots, x_m)
\Or
\TT(\psi(\underline x_1, \ldots,\underline x_m))
\\ & \Yields
\psi(x_1, \ldots, x_m)
\end{align*}
In the second implication, we apply a theorem of $\sigma$; in the fourth implication, we apply a property of the the truth predicate; in the fifth implication, we apply the validity of $(\tau + \sigma)^{(\alpha)}$. In \cref{section: truth}, we defined the truth predicate just for sentences without function symbols; however, if we extend this definition to sentences with function symbols by including the computation of values into verifications, then the conditional $$\Truth \Yields \TT(t(\underline x_1, \ldots, \underline x_m) = \underline{t(x_1, \ldots, x_m)})$$ is provable.

Thus, we have a proof in $(\tau+\sigma)^{(\Omega)}$ from the conjunction of the $\Sigma(L_{\omega \omega})$ formulas in the antecedent of the sequent $\AND \overline A, \AND E, \phi_1, \ldots, \phi_n \Proves \psi$, to $\psi$. If $\AND E$ is not a $\mathbf \Sigma(L_{\omega \omega})$ sentence, then we have a proof of $\phi_1 \And \ldots \And \phi_n \Yields \psi$, as desired. If $\AND E$ is $\mathbf \Sigma(L_{\omega \omega})$ formula, then we instead have a proof of $(\AND E) \And \phi_1 \And \ldots \phi_n \Yields \psi$, which can be turned into a proof of the desired conditional by applying the axiom $\Truth \Yields \AND E$ of $\sigma$.
\end{proof}

\begin{proposition}[base theory + every set is countable]
Let $\tau$ be a $\mathbf \Sigma (L_{\omega \omega})$ theory extending the base theory. Let $\phi \Yields \psi$ be a $\mathbf \Sigma (L_{\omega \omega})$ conditional. Assume that $\overline{\phi \Yields \psi}$ is assertible from $\tau$. Conclude that $\phi \Yields \psi$ is a theorem of $(\tau+ \sigma)^{(\Omega)}$.
\end{proposition}

\begin{proof}
There is a set $A$ of axioms of $\tau$, and a set $E$ of true $\mathbf \Sigma(L_{\infty \omega})$ formulas such that the sequent 
$\AND \overline A, \AND E \Proves \overline{\phi \Yields \psi}$ is derivable in $\lL \iI_{\infty \omega}$. It follows that the sequent $\AND \overline A, \AND E, \phi \Proves \psi$ is also derivable. By the theorem, the theory $(\tau + \sigma)^{(\Omega)}$ proves $\phi \Yields \psi$.
\end{proof}

\begin{proposition}[base theory + every set is countable]
Let $\tau$ be a $\mathbf \Sigma (L_{\omega \omega})$ theory extending the base theory. Let $ \psi$ be a $\mathbf \Sigma (L_{\omega \omega})$ sentence. Assume that $\psi$ is assertible from $\tau$. Conclude that $\Truth \Yields \psi$ is a theorem of $(\tau+ \sigma)^{(\Omega)}$.
\end{proposition}

\begin{proof}
If $\phi$ is assertible, then $\Truth \Yields \phi$ is also assertible. Apply the above corollary.
\end{proof}

\begin{definition}[base theory]
We extend the notion of a \underline{verification} $V$ to $\mathbf \Sigma (L_{\infty \omega})$ sentences as follows:
\begin{enumerate}\setcounter{enumi}{5}
\item If $\OR \Kappa$ is in $V$, then some element $\phi$ of $\Kappa$ is in $V$.
\item If $\AND \Kappa$ is in $V$, then each element $\phi$ of $\Kappa$ is in $V$.
\end{enumerate}
The \underline{truth} of a $\mathbf \Sigma(L_{\infty \omega})$ sentence is then defined, as before, as the existence of a verification containing that sentence.
\end{definition}

This extended truth predicate evidently satisfies the expected axioms for set conjunction and set disjunction:
\begin{enumerate}\setcounter{enumi}{5}
\item $ \TT(\OR \Kappa) \Equivalent \exists \phi \in \Kappa\: \TT( \underline \phi)$
\item $ \TT(\AND \Kappa) \Equivalent \forall \phi \in \Kappa \: \TT(\underline \phi)$
\end{enumerate}

The extended truth predicate is again a $\Sigma(L_{\omega \omega})$ formula, so we can reason about the truth of $\mathbf \Sigma (L_{\omega \omega})$ sentences using finitary, positivistic reasoning. In infinitary deductions it is convenient to establish the equivalence between $\TT(\OR \Kappa)$ and $\OR_{\phi \in \Kappa} \TT(\underline \phi)$, and likewise for conjunction. Bridging the gap between between $\forall \phi \in \Kappa\: \TT( \underline \phi)$ and  $\OR_{\phi \in \Kappa} \TT(\underline \phi)$ requires an infinitary assumption; the equivalence follows from a $\mathbf \Sigma(L_{\omega \omega})$ sentence essentially listing the elements of $\Kappa$, for example a conjunction of the content axiom for $\Kappa$ with the relevant atomic axioms.

\begin{lemma}[base theory]
 Let $\psi(x_1, \ldots, x_n)$ be a $\mathbf \Sigma(L_{\infty \omega})$ formula. There is a set $A_\psi$ of axioms of the base theory and a set $E_\psi$ of true $\mathbf \Sigma(L_{\infty \omega})$ sentences such that the sequents $\AND \overline A_\psi, \AND E_\psi, \psi(x_1, \ldots, x_n) \Proves \TT(\psi(\underline x_1, \ldots, \underline x_n))$ and $\AND \overline A_\psi, \AND E_\psi, \TT(\psi(\underline x_1, \ldots, \underline x_n)) \Proves \psi(x_1, \ldots, x_n)$ are provable in $\lL\iI_{\infty \omega}$.
\end{lemma}

\begin{proof}
The list of basic truth properties (1) - (7) is finite, so finitely many axioms of $\tau_0$ are sufficient to establish them. Let $A_\psi$ be the set of these axioms. Let $E_\psi$ consist of
all atomic axioms and content axioms for elements in the transitive closure of $\psi$, including content axioms for the conjunctions and disjunctions in $\psi$. As in the finitary case, we reduce $\TT(\psi(\underline x_1, \ldots, \underline x_n))$ to a $\mathbf \Sigma (L_{\omega \omega})$ formula that differs from $\psi(x_1, \ldots, x_n)$ only in that it expresses infinitary conjunctions and infinitary disjunctions using bounded quantifiers; we then use the content axioms and atomic axioms to show that this formula is equivalent to $\psi(x_1, \ldots x_n)$.
\end{proof}

\begin{theorem}[base theory + every set is countable + $(\tau+\sigma)^{(\Omega)}$ is valid]\label{block: assertibility properties}
Assume that $\tau$ is an $\mathbf \Sigma (L_{\omega \omega})$ theory extending the base theory. The assertibility predicate for $\tau$ respects the semantics of positivistic connectives:
\begin{enumerate}
\item $\AA_\tau(\phi \And \psi) \Equivalent \AA_\tau(\phi) \And \AA_\tau(\psi)$
\item $\AA_\tau(\phi \Or \psi) \Equivalent \AA_\tau(\phi) \Or \AA_\tau(\psi)$
\item $\AA_\tau(\AND_{\phi \in \underline{\Phi}} \phi) \Equivalent \forall \phi \in \Phi\: \AA_\tau(\phi)$
\item $\AA_\tau(\OR_{\phi \in \underline{\Phi}} \phi) \Equivalent \exists \phi \in \Phi\: \AA_\tau(\phi)$
\item $\AA_\tau(\exists v\: \psi) \Equivalent \exists a\:\AA_\tau(\psi^v_{\underline a})$
\item $\AA_\tau(\forall v \in \underline b\: \phi) \Equivalent \forall a \in b\: \AA_\tau(\psi^v_{\underline a})$
\end{enumerate}
The assertibility predicate satisfies modus ponens and the instantiation rule:
\begin{enumerate}\setcounter{enumi}{6}
\item $\AA_\tau(\phi) \And \AA_\tau(\phi \Yields \psi) \Yields \AA_\tau(\phi)$
\item $\AA_\tau(\forall v\: \phi) \Yields \AA_\tau(\phi^v_{\underline a})$
\end{enumerate}
Truths, tautologies, and the axioms of $\tau$ are assertible:
\begin{enumerate}\setcounter{enumi}{8}
\item $\TT(\phi) \Yields \AA_{\tau}(\phi)$, where the variable $\phi$ ranges over $ \mathbf \Sigma(L_{\omega \omega})$ sentences
\item $\DD_{\lL\iI}(\emptyset \Proves \phi) \Yields \AA_{\tau}(\phi)$, where the variable $\phi$ ranges over $\mathbf I (L_{\omega \omega})$ sentences
\item $\tau(\phi \Yields \psi) \Yields \AA_{\tau}(\overline{\phi \Yields \psi})$, where the variables $\phi$ and $\psi$ range over $\mathbf \Sigma(L_{\omega \omega})$ formulas.
\end{enumerate}
Finally, assertibility is valid in the following sense:
\begin{enumerate}\setcounter{enumi}{11}
\item $\AA_\tau(\phi) \Equivalent \TT(\phi)$, where $\phi$ ranges over $ \mathbf \Sigma(L_{\omega \omega})$ sentences
\item $\TT(\phi) \And \AA_\tau(\phi \Yields \psi) \Yields \TT(\psi)$, where the variables $\phi$ and $\psi$ range over $\mathbf \Sigma(L_{\omega \omega})$ formulas
\end{enumerate}

\end{theorem}

\begin{proof}
Property (1) is a simple consequence of rules for conjunction, as is property (3).

Property (2) is a consequence of cut-elimination. Fix a cut-free derivation of $\AND \overline A_0, \AND E_0, \Proves \phi \Or \psi$, for some set $A_0$ of axioms of $\tau$, and some set $E_0$ of true $\mathbf \Sigma(L_{\infty \omega})$ formulas. We will rework this derivation into a derivation of a sequent $\AND \overline{A_1}, \AND E_1 \Proves  \AA_\tau(\phi) \Or \AA_\tau(\psi)$, for some larger set $A_1 \supseteq A_0$ of axioms of $\tau$, and some larger set $E_1 \supseteq E_0$ of true $\mathbf \Sigma(L_{\infty \omega})$ sentences. Thus, we will conclude $\AA_\tau(\AA_\tau (\phi) \Or \AA_\tau(\psi))$. By corollary, the conditional $\Truth \Yields \AA_\tau (\phi) \Or \AA_\tau(\psi)$ is provable from $(\tau + \sigma)^{(\Omega)}$. By the validity of $(\tau+\sigma)^{(\Omega)}$, we will conclude $\AA_\tau(\phi) \Or \AA_\tau(\psi)$.

We distinguish the ``trunk'' of the fixed cut-free derivation of $\AND \overline A_0, \AND E_0, \Proves \phi \Or \psi$ as the final part of the proof tree that consists of sequents of the form $\Delta \Proves \phi \Or \psi$; we do not include sequents of this form that occur before sequents not of this form. We completely replace each subderivation of the fixed derivation that ends in some initial sequent of the trunk. An initial sequent of the trunk is of the form $\Delta \Proves \phi \Or \psi$ with each formula in $\Delta$ a substitution instance of a subformula of $\AND \overline A_0$ or of $\AND \overline E_0$. Thus, each formula in $\Delta$ is a partially closed axiom of $\tau$ or a conjunction of closed axioms of $\tau$ (we say that it is in $\Delta_A$), or it is a $\mathbf \Sigma(L_{\infty \omega})$ formula (we say that it is in $\Delta_E$). In this exposition, we assume that $\Delta$ does not contain a conjunction of closed axioms of $\tau$; in the case that $\Delta$ does contain a conjunction of closed axioms of $\tau$ the argument will need to be adjusted in a few inconsequential and messy ways.

Fix an initial sequent $\Delta \Proves \phi \Or \psi$ of the trunk. Since none of the immediately preceding sequents in the fixed cut-free derivation have $\phi \Or \psi$ in the succedent, it follows that that this initial sequent is produced by the right disjunction rule or the right weakening rule.  In the former case, we have a derivation of $\Delta \Proves \phi$ or of $\Delta \Proves \psi$, and in the latter case we have a derivation of both. Without loss of generality, we suppose that we have a derivation of $\Delta \Proves \psi$. We continue this subderivation, closing the axioms of $\tau$, and combining the closed axioms of $\tau$, and the $\mathbf \Sigma (L_{\infty \omega})$ formulas, into single formulas using the left conjunction rule, to obtain a derivation that looks quite like a witness to the assertibility of $\psi$, except that the $\mathbf \Sigma(L_{\infty \omega})$ conjunction may fail to be a sentence, and even if it is a sentence, it is not known to be true:
$$ \AND \overline{\Delta_A}, \AND \Delta_E(x_1, \ldots, x_n) \Proves \psi$$
We may formalize an argument that ``if there is a derivation of the sequent $ \AND \overline{\Delta_A}, \AND \Delta_E(x_1, \ldots, x_n) \Proves \psi$, and $\overline \Delta_A$ is a set of closures of axioms of $\tau$, and $\AND \Delta_E(x_1, \ldots, x_n)$, then the sentence $\psi$ is assertible''. The sentence ``there is a derivation of the sequent $ \AND \overline{\Delta_A}, \AND \Delta_E(x_1, \ldots, x_n) \Proves \psi$'' is a true $\mathbf \Sigma(L_{\infty \omega})$ sentence. The sentence ``$\overline \Delta_A$ is a set of closures of axioms of $\tau$'' is also a true $\mathbf \Sigma(L_{\infty \omega})$ sentence. Thus, there is a set of true $\mathbf \Sigma(L_{\infty \omega})$ sentences $E'$ and a set $A'$ of axioms of $\tau$, and derivations of the following sequents 
$$\AND \overline{A'}, \AND E',  \Delta_E(x_1, \ldots, x_n) \Proves \AA_\tau(\psi)$$
$$\AND \overline{A'}, \AND E',  \Delta_A,  \Delta_E(x_1, \ldots, x_n) \Proves \AA_\tau(\phi)\Or \AA_\tau(\psi)$$
Thus, for the given initial sequent $\Delta_A, \Delta_E(x_1, \ldots, x_n) \proves \phi \Or \psi$, we have obtained a derivation of the above sequent.

We now rework the fixed cut-free derivation of $\AND \overline A , \AND E \Proves \phi \Or \psi$ in the following way: we replace each derivation of an initial sequent of the trunk in the way described above; we replace each sequent $\Delta \Proves \phi \Or \psi$ of the trunk with the sequent $\Delta \Proves \AA_\tau(\phi) \Or \AA_\tau(\psi)$; we finally propagate the new assumptions in the initial sequents of the trunk to the conclusion, enlarging the conjunctions as we do so. Every step of the derivation above an initial segment of the trunk is according to $\lL\iI_{\infty \omega}$ because it is part of a derivation. Every other step whose output sequent is not in the trunk is according to $\lL\iI_{\infty \omega}$, because such a step has not been altered. Every step in the trunk is according to $\lL\iI_{\infty \omega}$, because the rules applied here in the fixed cut-free derivation are all left rules, and we have not introduces any free variables. The only non-initial steps in the trunk that have an input sequent not in the trunk apply the left implication rule; such implications are still valid because neither the new antecednent formulas nor the new succedent $\AA_\tau(\phi) \Or \AA_\tau(\psi)$ play a role. Thus, we obtain a derivation of $\AND \overline A_1, \AND  E_1 \Proves \AA_\tau(\phi) \Or \AA_\tau(\psi)$. The set $A_1$ consists of the axioms in $A_0$, together with the new axioms added at initial sequents of the trunk; likewise, $E_1$ consists of the axioms of $E_0$, together with the new true $\mathbf \Sigma(L_{\infty \omega})$ formulas added at initial sequents of the trunk. We conclude that $\AA_\tau(\phi) \Or \AA_\tau(\psi)$ is assertible, and as previously argued, that $\AA_\tau(\phi)$ or $\AA_\tau(\psi)$.

The above argument generalizes more or less directly to property (4). The argument for property (5) involves more changes, particularly when the vocabulary includes function symbols. We fix a cut-free derivation of $\AND \overline A_0, \AND E_0 \Proves \exists y\: \psi(y)$, and, as expected, we define the trunk to be the terminal part of this derivation that consists of sequents of the form $\Delta \Proves \exists y\: \psi(y)$. Each initial sequent of the trunk is produced by the right weakening rule or the right existential quantifier rule; in the later case, the preceding sequent is $\Delta \Proves \psi(t)$ for some term $t$ that may not be closed. Explicitly listing any free variables, we have a derivation of
$$\AND \overline \Delta_A, \AND \Delta_E(x_1, \ldots, x_n) \Proves \psi(t(x_1, \ldots, x_n))$$
We may formalize the argument that ``if $y = t(x_1, \ldots, x_n)$, and $\psi(t(x_1 ,\ldots, x_n))$, then $\psi(y)$''. Thus, there is a set of axioms $A''$ of $\tau$ so that we have a derivation of
$$\AND \overline{\Delta_A \union A''}, \AND \Delta_E (x_1, \ldots, x_n) \union \{y = t(x_1, \ldots x_n)\} \Proves \psi(y).$$
We may formalize an argument that ``if there is a derivation of $\AND \overline{\Delta_A \union A''}, \AND \Delta_E (x_1, \ldots, x_n) \union \{y = t(x_1, \ldots x_n)\} \Proves \psi(y)$, and $\overline \Delta_A$ is a set of closures of axioms of $\tau$, and $A''$ is a set of axioms of $\tau$, and $\AND \Delta_E (x_1, \ldots, x_n)$, and $y = t(x_1, \ldots, x_n)$, then the sentence $\psi(\underline y)$ is assertible''. We may also formalize an argument that ``there is an object $y$ such that $y = t(x_1, \ldots, x_n)$''.
Thus, there is a set of true $\mathbf \Sigma(L_{\infty \omega})$ sentences $E'$ and a set $A'$ of axioms of $\tau$, and derivations of the following sequents
$$\AND \overline{A'}, \AND E',  \Delta_E(x_1, \ldots, x_n), y = t(x_1, \ldots, x_n) \Proves \AA_\tau(\psi(\underline y)))$$
$$\AND \overline{A'}, \AND E',  \Delta_E(x_1, \ldots, x_n), \exists y\: y = t(x_1, \ldots, x_n) \Proves \exists y\: \AA_\tau(\psi(\underline y)))$$
$$\AND \overline{A'}, \AND E',  \Delta_E(x_1, \ldots, x_n) \Proves \exists y\: \AA_\tau(\psi(\underline y)))$$
The rest of the argument follows the pattern of the disjunctive cases.

For property (6), assume that $\forall v\in \underline b \: \psi$ is assertible from $\tau$. If $a$ is an element of $b$, then $\psi^v_{\underline a}$ is assertible; we simply extend the derivation by adding $\underline a \in \underline b$ to the antecedent, and instantiating. Conversely, assume that for each element $a$ in $b$, the formula $\psi^v_{\underline a}$ is assertible from $\tau$. We combine the given derivations into a derivation whose consequent is $\AND_{a \in b} \psi^v_{\underline a}$; we conclude that $\forall v \in \underline b\: \psi$ is assertible, by adding to the antecedent the true $\mathbf \Sigma(L_{\infty \omega})$ sentences listing the elements of $b$.

Properties (7) and (8) follow simply from the rules of $\lL \iI_{\infty \omega}$. Properties (9), (10), and (11), follow from the definition of assertibility. For property (12), assume $\AA_\tau(\phi)$; it follows that the conditional $\Truth \Yields \phi$ is a theorem of $(\tau+ \sigma)^{(\Omega)}$, so by the validity of $(\tau + \sigma)^{(\Omega)}$, we conclude that $\phi$ is true. Property (13) follows by combining property (7) and property (12).
\end{proof}

\section{Philosophical remarks}\label{section: remarks}

The conservative response to objections that a formula with free variables implicitly expresses a universal proposition if it is meaningful at all is to interpret the free variables as syntactic variables. If each free variable in a positivistic proof is replaced with a constant symbol, the proof remains valid at each step. The axioms of a theory are then also interpreted schematically; we must include all closed instances of deep inference. 

The drawback of this conservative interpretation is that it undermines the impact of theorems that include free variables. Such a theorem may be interesting because its closed instances are \emph{uniformly} provable. Each instance of its proof is positivistically acceptable, but this observation is not itself expressible as an inference under the conservative interpretation. Indeed, each instance of the observation itself is provable, but what is interesting is that these instances are uniformly provable; and we have come full circle.

This difficulty is largely an artifact of our formalization of positivistic reasoning within the framework of first-order logic. Procedures, not formulas, are the elementary constituents of a mathematical universe in the positivistic approach. We imagine a register machine whose register contents may change over the course of a machine procedure, or indeed, over the course of an argument. At each step of the argument, we predict that a certain procedure will halt \emph{after} the sequence of preceding procedures, \emph{on the basis} that the sequence of preceding procedures halts.

We may analyze the operation of the register machine modally. Each machine procedure is a binary relation on the space of machine states. We interpret each positivistic formula $\phi(\overline x)$ as the procedure that halts on its initial state if $\phi(\overline x)$ can be verified for the register contents of that state, and that otherwise does not halt. Ignoring deep inference, we might then interpret each conditional $\phi(\overline x) \Yields \psi(\overline x)$ as the rule of inference the predicts $\psi(\overline x)$ after $\phi(\overline x)$. However, our reasoning may also include a procedure that may halt on a state other than its initial state.

Here the discussion splits into two threads. First, we discuss the extent to which universal quantification is implicit in this dynamic reasoning. Second, we formalize this dynamic reasoning in a way that incorporates deep inference and mathematical practice.

Recall that existential quantification on $x$ is defined in terms of a procedure that is presumed to change the contents of register $x$ arbitrarily. This presumed behavior is expressed by the logical axiom $\phi(t) \Yields \exists x\: \phi(x)$. If our argument assumes $\phi(x)$, and infers $\psi(x)$, then $x$ is essentially a schematic variable that denotes the object in register $x$, which remains unchanged over the course of the argument. In contrast, if our argument assumes the existential quantifier operation for $x$ followed by $\phi(x)$, and infers $\psi(x)$, then $x$ is essentially a universally quantified variable, since we have inferred $\psi(x)$ about an arbitrarily chosen object that satisfies $\phi(x)$. That our inference carries a universal sense does not contradict the positivistic rejection of universal quantification over potentialist totalities because it is an inference, not a proposition. We are simply inferring $\psi$ from $\phi$ for an object that was chosen arbitrarily, i. e., produced by a specific procedure.

We may summarize this point by saying that though the positivistic approach excludes the possibility of a procedure that checks whether a given predicate holds for every object in the universe, it \emph{requires} a procedure that produces arbitrary objects. It is the procedure that axiomatically produces arbitrary objects: to infer $\phi(x)$ for arbitrary $x$, is to infer $\phi(x)$ for the object produced by this procedure, and to infer that this procedure produces arbitrary objects is to infer that this procedure will produce an object equal to the object produced by this procedure. Our axioms for the existential quantifier express principles appropriate to this intuition.

In particular, the correctness principle of an intentional theory is the inference that if an axiom of that theory was produced arbitrarily, i. e., by the existential quantifier procedure, and the assumption of that axiom is true, then the conclusion of that axiom is true. Corollary \ref{block: correctness} shows that there are strong theories that prove this inference about themselves. The substitution principle that we can replace an arbitrary object in an argument with a constant symbol is a provable theorem; it also is formulated in terms of the existential quantifier procedure: if an arbitrary free variable is replaced with an arbitrary constant symbol in an arbitrary proof, then the result is also a proof. Thus, we have a device for formulating universal principles without admitting universal propositions.

The major subtlety of dynamic reasoning is the treatment of state-preserving machine procedures such as formulas. For example, an argument might assume $\phi$, then infer $\chi$, and then infer $\psi$, leading us to claim that assuming $\phi$, we can infer $\psi$. However, this summary presumes that the procedure $\chi$ does not alter the machine state, a presumption that is not entirely trivial. After all, the verification of $\chi$ generally involves temporarily storing data in various registers. We may address this difficulty by allowing hypothetical reasoning, which replaces a terminal part of the sequence of procedures in a deduction, rather than simply adding to it. We may also simply adopt the convention that the occurrence of intermediate formulas is implicit in the statement of a theorem. In either case the direct formalization of the deductive system is somewhat alien to mathematical practice.

Instead, we formalize dynamic reasoning in a way that synthesizes aspects of potentialism, of deep inference, and of mathematical practice. Any true $\Sigma$ formula is true in some initial segment of the universe, in which classical reasoning is positivistically valid. So, we take the $\Sigma$ reflection principle as a basic rule of inference, and we relativize the standard Hilbert system for classical deduction to transitive sets. We similarly relativize the axioms of the given positivistic theory to transitive sets. Finally, we formalize the dynamic aspect of the reasoning using the familiar phrase ``let $b_1$ be...", to allow the introduction of constant symbols after the deduction of an existential sentence.

To simplify the ensuing discussion, we assume that the signature includes no function symbols of positive arity; we allow the negation connective $\neg$ to occur in $\Sigma$ formulas, provided that it does not occur above unbounded quantifiers; and we express the negated predicate symbols $\not \in$ and $\neq$ using this connective. We recall that when $\phi(x_1, \ldots, x_n)$ is a classical first-order formula, its relativization $\phi^W(x_1, \ldots, x_n)$ to a set $W$ is defined to be the result of replacing each existential quantifier $\exists v \:$ by the bounded quantifier $\exists v \in W\:$, and likewise for universal quantifiers. By convention, even bounded quantifiers are replaced in this way, so that $\forall y \in x \: \phi$ becomes $\forall y \in W\: y \not \in x \Or \phi$. Our notation for the relative universal closure $\overline{ \phi (x_1, \ldots, x_n)}^W$ may be read directly; it abbreviates the formula $\forall x_1 \in W\: \cdots \forall x_n \in W\: \phi^W(x_1, \ldots, x_n)$.

\begin{definition}
Let $T$ be an extensional positivistic theory, i. e., a set of positivistic conditionals. The deductions of \underline{the system $DT$} consist of $\Sigma$ sentences. For the rules of inference of $DT$ listed below, $\phi$, $\chi$, and $\psi$ denote arbitrary $\Sigma$ formulas, and $c_1, \ldots, c_n$ denotes a nonempty list of constant symbols that includes those of $\phi$ and $\chi$. The notation $\Trans(X, \overline c)$ abbreviates the formula
$(c_1 \in X) \And(c_2 \in X) \And \cdots \And (c_n \in X) \And (\forall w \in X \: \forall v \in w\: v \in X)$.

\begin{enumerate}
\item \textbf{Nonlogical axioms}. If $\phi(x_1, \ldots, x_m) \Yields \psi(x_1, \ldots, x_m)$ is an axiom of $T$, from $\Trans(C, \overline c)$, deduce $\forall x_1 \in C\: \ldots \forall x_m \in C\: \neg \phi^C(x_1, \ldots, x_m) \Or \psi(x_1, \ldots, x_m)$.
\item \textbf{Logical axioms}. If $\phi$ is an axiom of the standard Hilbert system for classical first-order logic with equality, from $\Trans(C, \overline c)$, deduce $\overline{\phi}^C$.
\item \textbf{Modus ponens}. From $\Trans(C, \overline c )$, and $\overline{\phi}^C$, and $\overline{\neg \phi \Or \chi}^{C}$, deduce $\overline{\chi}^C$.
\item \textbf{Downward reflection.} From $\phi$, deduce $\exists X\: \Trans(X, \overline c) \And \phi^X$.
\item \textbf{Upward reflection.} From $\exists X\: \Trans(X, \overline c) \And \phi^X$, deduce $\phi$.
\item \textbf{Conjunction introduction.} From $\phi$ and $\chi$, deduce $\phi \And \chi$.
\item \textbf{Conjunction elimination.} From $\phi \And \chi$ or $\chi \And \phi$, deduce $\phi$.
\item \textbf{Universal instantiation.} From $\Trans(C, \overline c)$ and $\overline{\phi}^C$, deduce $\overline{\phi(x/c_1)}^C$.
\item \textbf{Existential generalization.} From $\phi(X/C)$, deduce $\exists X \: \phi$.
\item \textbf{Existential instantiation.}  From $\exists X \: \phi$, deduce $\phi(X/B)$, where $B$ is a new constant symbol, i. e., a constant symbol that is not in the signature and that does not occur previously in the deduction.
\end{enumerate}
A $\Sigma$ sentence $\psi$ is deducible just in case it has a deduction from the $\Sigma$ formula $\Truth$. A $\Sigma$ conditional $\phi(y_1, \ldots, y_n) \Yields \psi(y_1, \ldots, y_n)$ is deducible just in case for new constant symbols $b_1, \ldots, b_n$, the formula $\psi(b_1, \ldots, b_n)$ can be deducible from $\phi(b_1, \ldots, b_n)$.
\end{definition}

Following standard mathematical practice, we gloss the introduction of a new constant symbol using the word ``let''. Mechanically, this introduction refers to an execution of the existential quantifier procedure, whose output persists through the proof as the value of the new constant symbol. Explicitly, in an application of the existential instantiation inference rule, after we check $\exists X\: \phi$, we infer that the procedure that looks for such an object and stores it in register $B$ halts, and then we infer $\phi(X/B)$, so that we can later apply inferences that assume $\phi(X/B)$. We gloss an initial assumption $\phi(b_1, b_2, \ldots, b_m)$ in the deduction of a $\Sigma$ conditional $\phi(\overline x) \Yields \psi(\overline x)$ in just the same way; thus, such an initial assumption technically corresponds to a finite sequence of procedures, rather than a single procedure.

\begin{theorem}[$\ZFC$]\label{block: DT}
Assume that $\chi\Yields \chi'$ is provable from a positivistic theory $T$. Then, this conditional is deducible in the system $DT$ just defined.
\end{theorem}

\begin{proof}
It's sufficient to establish this theorem when $\chi \Yields \chi'$ is an instance of deep inference of some axiom of $T$, or some axiom of positivistic logic.

Assume that $\chi(\overline y) \Yields \chi'(\overline y)$ is an instance of deep inference of an axiom of positivistic logic, so it is a logical validity. It follows that $\neg \chi(\overline y) \Or \chi'(\overline y)$ is derivable in the standard Hilbert system. We now construct the desired deduction as follows: Assume  $\chi(\overline b)$. Let $B$ be a transitive set that contains the constants of $\chi$, $\chi'$, and $\overline b$, such that $\chi(\overline b)^B$. Derive $(\neg \chi(\overline b) \Or \chi'(\overline b))^B$. Apply modus ponens to derive $\chi'(\overline b)^B$. Conclude $\chi'(\overline b)$ by upward reflection.

Assume that $\chi(\overline y) \Yields \chi'(\overline y)$ is an instance of deep inference of $\phi(\overline x) \Yields \psi(\overline x)$, an axiom of $T$. Write $\overline c$ for the constant symbols of $\chi$ and $\chi'$. We now construct the desired deduction as follows: Assume $\chi(\overline b)$. Let $B$ be a transitive set that contains $\overline b$ and $\overline c$, such that $\chi(\overline b)^B$. The constant symbols of $\phi$ must be among the constant symbols of $\chi$, so derive $\forall x_1 \in B\: \cdots \forall x_n \in B\: \neg \phi^B(x_1, \ldots, x_m) \Or \psi(x_1, \ldots, x_m)$ by the nonlogical axioms rule. Applying conjunction introduction and downward reflection, obtain a transitive set $B'$ that contains $B$, $\overline b$, and $\overline c$, such that
$$ \left(\chi(\overline b)^B \And \Trans(B, \overline b, \overline c) \And \forall x_1 \in B\: \cdots \forall x_n \in B\: \neg \phi^B(\overline x) \Or \psi(\overline x)\right)^{B'}.$$
Since deep inference is logically valid, reasoning in the Hilbert system inside $B'$, derive $\chi'(b_1, \ldots, b_n)^{B'}$. Conclude $\chi'(b_1, \ldots, b_n)$ by upward reflection.
\end{proof}

\begin{corollary}[$\ZFC$]
Assume that $\chi(\overline y) \Yields \chi'(\overline y)$ is provable from a positivistic theory $T$. Let $\overline c$ be a sequence of constants symbols that includes those of $\chi$. Then, the $\Sigma$ formula $$\forall y_1 \in C\: \cdots \forall y_n \in C\: \neg \chi^{C}(y_1, \ldots, y_n) \Or \chi'(y_1, \ldots, y_n)$$ is deducible from $\Trans(C, \overline c)$ in the system $DT$, for any new constant symbol $C$.
\end{corollary}

\begin{proof}
The classical first-order formula $\Trans(Y, \overline c)$ implies $\forall y_1 \in Y\: \cdots \forall y_n \in Y\: \neg \chi^Y(y_1, \ldots, y_n) \Or \chi(y_1, \ldots, y_n)$; this follows by induction on the complexity of $\chi$ because $\chi$ is a $\Sigma$ formula. The induction hypothesis affirms the implication for all tuples of variables $\overline y$. As an example, in the bounded universal quantifier case, we assume $\Trans(Y, \overline c)$ and let $y_1, \ldots, y_n \in Y$ be such that $(\forall v \in t\: \chi)^Y$, for some term $t$ that is necessarily either a constant symbol or a variable among $\overline y$. In either case, we have $t \in Y$. The formula $(\forall v \in t\: \chi)^Y$ is equal to $\forall v \in Y\: \neg v \in t \Or \chi^Y$, and since $Y$ is transitive, this formula implies $\forall v \in t\: \chi^Y$. By the induction hypothesis, $\Trans(Y, \overline c)$ implies $\forall y_1 \in Y \: \ldots \forall y_n \in Y\: \forall v \in Y\: \neg \chi^Y \Or \chi$, so we conclude $\forall v \in t\: \chi$.

Thus, the theory $T$ logically implies the conditional
$$ \Trans (Y, \overline c) \Yields \forall y_1 \in Y\: \cdots \forall y_n \in Y\: \sim \chi^Y(y_1, \ldots, y_n) \Or \chi'(y_1, \ldots, y_n),$$
where the notation $\sim$ abbreviates taking the logical complement of a $\Delta_0$ formula.
By the completeness theorem, the theory $T$ proves this conditional; and by \cref{block: DT}, the system $DT$ derives $\forall y_1 \in C\: \cdots \forall y_n \in C\: \sim \chi^C(y_1, \ldots, y_n) \Or \chi'(y_1, \ldots, y_n)$ from $\Trans (C, \overline c)$, for any new constant symbol $C$. We may replace the notation $\sim$ with the genuine negation connective $\neg$ reasoning within the Hilbert system relativized to a large transitive set.
\end{proof}

\textbf{Summary.} The admissibility of free variables in positivistic proof derives from two aspects of positivistic reasoning. First, a positivistic inference expresses that the consequent procedure will halt after the antecedent procedure, if the antecedent procedure does halt, so data can be passed from one procedure to the next. Second, procedures are innately nondeterministc: a priori any procedure may execute and halt in any number of ways. There is a procedure $E$ that is presumed to produce arbitrary objects; the principle that $\phi$ is true of each object may be expressed by the inference ``after $E$ produces an object $x$, the procedure $\phi(x)$ will halt''.

A theory that proves its own correctness principle, such as in \cref{block: correctness}, is self-affirming in the sense that it proves that if an object is produced by the procedure $E$, and that object is the application of an axiom whose antecedent is true, then its consequent is also true. Such a theory may be of arbitrarily high consistency strength; there is no contradiction with G\"odel's second incompleteness theorem because the full induction principle is not positivistically expressible.

\begin{figure}\label{figure: rule step}

\begin{enumerate}[(1)]\setcounter{enumi}{-3}
\item $\phi \Yields \phi$
\item $\deduce{\phi \INFER \psi}{\phi \INFER \chi & \chi \INFER \psi}$
\item $\infer{\phi^a_t \Yields \psi^a_t}{\phi \Yields \psi}$
\item $\phi \Yields \Truth$
\item $\Falsehood \Yields \psi$
\item $\phi \And \psi \Yields \phi$
\item $\phi \And \psi \Yields \psi$
\item $\deduce{\phi \Yields \chi \And \psi} {\phi \INFER \chi & \phi \INFER \psi}$
\item $\phi \Yields \phi \Or \psi$
\item $\psi \Yields \phi \Or \psi$
\item $\deduce{\phi \Or \chi \INFER \psi}{\phi \INFER \psi & \chi \INFER \psi }$
\item $(\phi \Or \psi) \And \chi \Yields (\phi \And \chi) \Or (\psi \And \chi)$
\item $\phi^x_t \Yields \exists x \: \phi$
\item $\infer{\exists y \: \phi \Yields \psi}{\phi^y_a \Yields \psi}$ \hfill for $a$ not free in $\phi$ or $\psi$
\item $\chi \And \exists y\: \phi \Yields \exists y\: (\chi \And \phi)$
\item $ \forall x \in s \: \phi \Yields t \not \in s \Or \phi^x_a$ 
\item $ \deduce{\psi \Yields \forall y \in s\: \phi}{\psi \Yields a \not \in s \Or \phi^y_a}$ \hfill for $a$ not free in $\phi$ or $\psi$
\item $\forall y \in s \: (\chi \Or \phi) \Yields \chi \Or (\forall y \in s \: \phi)$
\item $\Truth \Yields a \in b \Or a \not \in b$
\item $a \in b \And a \not \in b \Yields \Falsehood$
\item $\Truth \Yields a  = b \Or a \neq b$
\item $a = b \And a \neq b \Yields \Falsehood$
\item $\Truth \Yields a = a$
\item $a = b \Yields b =a$
\item $a=b \And b =c \Yields a= c$
\item $s = t \And \phi^a_t \Yields \phi^a_s$
\end{enumerate}

\caption{a complete calculus of conditionals}
(bound variables distinct from free variables)
\end{figure}

\begin{figure}\label{figure: conditional step}

\begin{enumerate}[(1)]
\item $\phi \Yields \Truth$
\item $\Falsehood \Yields \psi$
\item $\phi \And \psi \Yields \phi$
\item $\phi \And \psi \Yields \psi$
\item $\phi \Yields \phi \And \phi$
\item $\phi \Yields \phi \Or \psi$
\item $\psi \Yields \phi \Or \psi$
\item $\psi \Or \psi \Yields \psi$
\item $(\phi \Or \psi) \And \chi \Yields (\phi \And \chi) \Or (\psi \And \chi)$
\item $\phi^x_t \Yields \exists x \: \phi $
\item $\exists y\: \psi \Yields \psi$ 
\item $\chi \And \exists y\: \phi \Yields \exists y\: (\chi \And \phi)$
\item $ t \in s \And \forall x \in s \: \phi \Yields \phi^x_t$
\item $ \psi \Yields \forall y \in s\: (\psi \And y \in s)$ 
\item $\forall y \in s \: (\chi \Or \phi) \Yields \chi \Or (\forall y \in s \: \phi)$ 
\item $\Truth \Yields s \in t \Or s \not \in y$
\item $s \in t \And s \not \in t \Yields \Falsehood$
\item $\Truth \Yields s = t \Or s \neq t$
\item $s = t \And s \neq t \Yields \Falsehood$
\item $\Truth \Yields t = t$
\item $s = t \Yields t =s$
\item $r=s \And s =t \Yields r = t$
\item $s = t\And \phi^a_s \Yields \phi^a_t$ 
\end{enumerate}

\caption{a complete class of logical axioms for positivistic proof}
(bound variables distinct from free variables)
\end{figure}
\quad

\begin{figure}\label{figure: logical axioms}

\begin{enumerate}[(1)]
\item $\phi \Yields \Truth$
\item $\Falsehood \Yields \psi$
\item $\phi \And \psi \Yields \phi$
\item $\phi \And \psi \Yields \psi$
\item $\phi \Yields \phi \And \phi$
\item $\phi \Yields \phi \Or \psi$
\item $\psi \Yields \phi \Or \psi$
\item $\psi \Or \psi \Yields \psi$
\item $(\phi \Or \psi) \And \chi \Yields (\phi \And \chi) \Or (\psi \And \chi)$
\item $\phi \Yields \exists x \: \phi $
\item $\exists y\: \psi \Yields \psi$ \hfill for $y$ not free in $\psi$
\item $\chi \And \exists y\: \phi \Yields \exists y\: (\chi \And \phi)$ \hfill for $y$ not free in $\chi$
\item $ x \in z \And \forall x \in z \: \phi \Yields \phi$
\item $ \psi \Yields \forall y \in z\: (\psi \And y \in z)$ \hfill for $y$ not free in $\psi$
\item $\forall y \in z \: (\chi \Or \phi) \Yields \chi \Or (\forall y \in z \: \phi)$ \hfill for $y$ not free in $\chi$
\item $\Truth \Yields x \in y \Or x \not \in y$
\item $x \in y \And x \not \in y \Yields \Falsehood$
\item $\Truth \Yields x = y \Or x \neq y$
\item $x = y \And x \neq y \Yields \Falsehood$
\item $\Truth \Yields x = x$
\item $x = y \Yields y =x$
\item $x=y \And y =z \Yields x = z$
\item $x = y\And \phi^z_x \Yields \phi^z_y$ \hfill for $x$ and $y$ substitutable for $z$ in $\phi$
\end{enumerate}

\caption{the logical axioms of positivistic proof}
\end{figure}
\quad

\begin{figure}\label{figure: LK}

\begin{center} Introduction Rules  \end{center}
$
\begin{tabular}{ccc}
\\
\infer{\Falsehood \proves \emptyset}{} & & \infer{\emptyset \Proves \Truth}{}
\\
\\
\\
\infer{\Gamma,  \neg \phi  \Proves \Delta }{ \Gamma \Proves \Delta, \phi} & & \infer{ \Gamma \proves \Delta, \neg \phi}{\Gamma, \phi \Proves \Delta }
\\
\\
\\
\infer{\Gamma, \AND \Kappa \Proves \Delta}{(\exists \phi \in \Kappa) \quad \Gamma , \phi \Proves \Delta} & \qquad\qquad&     \infer{\Gamma \Proves \Delta, \AND \Kappa}{(\forall \phi \in \Kappa) \quad \Gamma \Proves \Delta, \phi} \\ 
\\ \\
\infer{\Gamma, \OR \Kappa \Proves \Delta}{(\forall \phi \in \Kappa) \quad \Gamma , \phi \Proves \Delta} & &       \infer{\Gamma \Proves \Delta, \OR \Kappa}{(\exists \phi \in \Kappa) \quad \Gamma \Proves \Delta, \phi} \\ \\ \\
\infer{ \Gamma, \forall v \: \phi \Proves \Delta}{(\exists t) \quad \Gamma, \phi^v_t \Proves \Delta} & & \infer{\Gamma \Proves \Delta , \forall v\: \phi}{(\exists w \not \in \mathrm{Free}(\Gamma, \Delta)) \quad \Gamma \proves \Delta, \phi^v_w}
\\ \\ \\
\infer{ \Gamma, \exists v \: \phi \Proves \Delta}{(\exists w \not \in \mathrm{Free}(\Gamma, \Delta)) \quad \Gamma, \phi^v_w \Proves \Delta} & & \infer{\Gamma \Proves \Delta , \exists v\: \phi}{(\exists t) \quad \Gamma \proves \Delta, \phi^v_t}
\\ \\ \\
\infer{\Gamma, \phi \Proves \Delta}{ \Gamma \Proves \Delta} & & \infer{\Gamma \Proves \Delta, \phi}{\Gamma \Proves \Delta}
\end{tabular}
$ \\
\quad \\
\begin{center} Structural Rule\end{center} 
\quad\\
$$\infer{\Gamma' \Proves \Delta'}{\Gamma \Proves \Delta}$$
($\Gamma'$ consists of the same formulas as $\Gamma$, and $\Delta'$ consists of the same formulas as $\Delta$) \\
\quad\\

\begin{center} Cut Rule \end{center} 
\quad\\
$$ \infer {\Gamma, \Gamma'  \Proves \Delta , \Delta'}{\Gamma \proves \Delta, \phi & \Gamma', \phi \proves \Delta'}$$

\caption{the system $\lL\kK_{\infty \omega}$}
($\Gamma$, $\Delta$ finite)
\end{figure}

\begin{figure}\label{figure: LI}

\begin{center} Introduction Rules  \end{center}
$
\begin{tabular}{ccc}
\\
\infer{\Falsehood \Proves \emptyset}{\quad} & & \infer{\emptyset \Proves \Truth}{\quad}
\\ \\ \\
\infer{\Gamma, \Gamma', \phi \Yields \psi \Proves \Delta}{ \Gamma \Proves \phi & \Gamma', \psi \Proves \Delta} & & \infer{ \Gamma \proves \phi \Yields \psi}{\Gamma, \phi \Proves  \psi}
\\
\\
\\
\infer{\Gamma, \AND \Kappa \Proves \Delta}{(\exists \phi \in \Kappa) \quad \Gamma , \phi \Proves \Delta} & \qquad\qquad&     \infer{\Gamma \Proves 
\AND \Kappa}{(\forall \phi \in \Kappa) \quad \Gamma \Proves 
\phi} \\ 
\\ \\
\infer{\Gamma, \OR \Kappa \Proves \Delta}{(\forall \phi \in \Kappa) \quad \Gamma , \phi \Proves \Delta} & &       \infer{\Gamma \Proves \OR \Kappa}{(\exists \phi \in \Kappa) \quad \Gamma \Proves  \phi} \\ \\ \\
\infer{ \Gamma, \forall v \: \phi \Proves \Delta}{(\exists t) \quad \Gamma, \phi^v_t \Proves \Delta} & & \infer{\Gamma \Proves  \forall v\: \phi}{(\exists w \not \in \mathrm{Free}(\Gamma)) \quad \Gamma \proves \phi^v_w}
\\ \\ \\
\infer{ \Gamma, \exists v \: \phi \Proves \Delta}{(\exists w \not \in \mathrm{Free}(\Gamma, \Delta)) \quad \Gamma, \phi^v_w \Proves \Delta} & & \infer{\Gamma \Proves  \exists v\: \phi}{(\exists t) \quad \Gamma \proves  \phi^v_t}
\\ \\ \\
\infer{\Gamma, \phi \Proves \Delta}{ \Gamma \Proves \Delta} & & \infer{\Gamma \Proves  \phi}{\Gamma \Proves \emptyset}
\end{tabular}
$ \\
\quad \\
\begin{center} Structural Rule\end{center} 
\quad\\
$$\infer{\Gamma' \Proves \Delta'}{\Gamma \Proves \Delta}$$
($\Gamma'$ consists of the same formulas as $\Gamma$, and $\Delta'$ consists of the same formulas as $\Delta$) \\
\quad\\

\begin{center} Cut Rule \end{center} 
\quad\\
$$ \infer {\Gamma, \Gamma'  \Proves  \Delta}{\Gamma \proves \phi & \Gamma', \phi \proves \Delta}$$

\caption{the system $\lL \iI_{\infty \omega}$}
($\Gamma$ finite, $\Delta$ singleton or empty)

\end{figure}

\begin{figure}\label{figure: LSigma}

\begin{center} Introduction Rules  \end{center}
\vspace{0.2cm}

$$\infer{\Falsehood \Proves \emptyset}{\quad} \qquad\qquad \infer{\emptyset \Proves \Truth}{\quad}$$\\

$$\infer{\Gamma, \phi \And \psi \Proves \Delta}{ \Gamma , \phi \Proves \Delta} \qquad\qquad\infer{\Gamma, \phi \And \psi \Proves \Delta}{ \Gamma , \psi \Proves \Delta}  \qquad\qquad  \infer{\Gamma \Proves 
\Delta, \phi \And \psi}{ \Gamma \Proves 
\Delta, \phi & \Gamma \Proves \Delta, \psi}$$
\\ 
$$\infer{\Gamma, \phi \Or \psi \Proves \Delta}{\Gamma , \phi \Proves \Delta & \Gamma , \psi \Proves \Delta} \qquad\qquad \infer{\Gamma\Proves \Delta, \phi \Or \psi }{\Gamma  \Proves \Delta, \phi} \qquad\qquad    \infer{\Gamma\Proves \Delta, \phi \Or \psi }{\Gamma  \Proves \Delta, \psi}$$
\\ 
$$\infer{ \Gamma, \exists v \: \phi \Proves \Delta}{\Gamma, \phi^v_w \Proves \Delta} \qquad\qquad \infer{\Gamma \Proves  \Delta, \exists v\: \phi}{\Gamma \proves \Delta, \phi^v_t}$$
\\
$$\infer{ \Gamma, \forall v \in s \: \phi \Proves \Delta}{\Gamma, t \not \in s \Or \phi^v_t \Proves \Delta} \qquad\qquad \hspace{-30pt}\infer{\Gamma \Proves \Delta, \forall v \in s\: \phi}{\Gamma \proves \Delta,  w \not \in s \Or\phi^v_w}$$
\\ 
$$\infer{\Gamma, \phi \Proves \Delta}{ \Gamma \Proves \Delta} \qquad\qquad \infer{\Gamma \Proves  \Delta, \phi}{\Gamma \Proves \Delta}$$

\vspace{0.5cm}

\begin{center} Structural Rule\end{center} 

$$\infer{\Gamma' \Proves \Delta'}{\Gamma \Proves \Delta}$$

\vspace{5mm}
\begin{center} Cut Rule \end{center}

$$\infer{\Gamma, \Gamma' \Proves \Delta, \Delta'}{   \Gamma \Proves \Delta, \phi & \Gamma', \psi \Proves \Delta'}$$

\caption{the system $\lL \mathbf \Sigma_{\omega \omega}(\tau)$ for a $\mathbf \Sigma(L_{\omega \omega})$ theory $\tau$}
 \vspace{4mm}
\begin{flushleft}
For the cut rule, $\phi \Yields \psi$ should be an axiom of $\tau$. For the structural rule, $\Gamma'$ should consist of the same formulas as $\Gamma$, and $\Delta'$ should consist of the same formulas as $\Delta$. For left existential quantification rule, and for the right universal quantification rule, $w$ should not be free in $\Gamma$ or $\Delta$.
\end{flushleft}
\end{figure}

\begin{figure}\label{figure: PRS}
An alternate axiomatization of $\PRS$:
\begin{enumerate}
\item $(\forall t \in a\: t \in b) \And (\forall t \in b\: t \in a) \Yields a =b$
\item $\exists t\: t \in x \Yields \exists t \in x\: \forall s \in x\: s \not \in t$
\item $z \in \{x,y\} \Equivalent z = x \Or z = y$
\item $z \in \Union x \Equivalent \exists y \in x: z \in y$
\item $y \in \{ t \in x\suchthat \phi(t,\overline z)\} \Equivalent y \in x \And \phi(y,\overline z)$
\item $y \in \{F(t,\overline z)\suchthat t \in x \} \Equivalent \exists t \in x\: y = F(t,\overline z)$
\item $\Truth \Yields P(x_1, \ldots, x_n) = x_i$
\item $\Truth \Yields K(x_1, \ldots, x_m) = F(G_1(x_1, \ldots, x_m), \ldots, G_n(x_1, \ldots, x_m))$
\item $\Truth \Yields F(x, y) = R(\{F(t, y) \suchthat t \in x\},x,y).$
\end{enumerate}

The supplementary axioms:
\begin{enumerate}
\item $\forall s \in x\: \exists t\: \phi(s,t, \overline z) \Yields  \exists y\: \forall s \in x \: \exists t \in y\: \phi(s, t, \overline z)$
\item $\Truth \Yields \exists f\: \mathrm{Bij}(f) \And \mathrm{Ord}(\mathrm{dom}(f)) \And \mathrm{ran}(f) = x$.

\end{enumerate}

\caption{the initial axioms of our base theory}
\end{figure}

\end{document}